\documentclass[10pt]{amsart}

\usepackage{bbm}
\usepackage{eucal}
\numberwithin{equation}{subsection}
\usepackage{xcolor}
\definecolor{winered}{rgb}{0.8,0,0}
\definecolor{deepblue}{rgb}{0,0,0.8}
\usepackage[colorlinks=true]{hyperref}
\hypersetup{linkcolor=black, citecolor=black}
\usepackage{amsmath}
\usepackage{amsthm}
\usepackage{amssymb}
\usepackage{mathrsfs}
\usepackage{tikz-cd}
\newtheorem{thm}{Theorem}[subsection]
\newtheorem{prop}[thm]{Proposition}
\newtheorem{cor}[thm]{Corollary}

\newtheorem{lem}[thm]{Lemma}
\theoremstyle{definition}
\newtheorem{df}[thm]{Definition}

\newtheorem{exm}[thm]{Example}
\newtheorem{const}[thm]{Construction}

\theoremstyle{remark}

\newcommand{\bC}{\mathbf{C}}

\newcommand{\A}{\mathbb{A}}

\newcommand{\E}{\mathbb{E}}

\newcommand{\G}{\mathbb{G}}

\newcommand{\N}{\mathbb{N}}
\renewcommand{\P}{\mathbb{P}}
\newcommand{\Q}{\mathbb{Q}}

\newcommand{\Z}{\mathbb{Z}}

\newcommand{\cA}{\mathcal{A}}

\newcommand{\cC}{\mathcal{C}}
\newcommand{\cD}{\mathcal{D}}
\newcommand{\cE}{\mathcal{E}}
\newcommand{\cF}{\mathcal{F}}
\newcommand{\cG}{\mathcal{G}}

\newcommand{\cK}{\mathcal{K}}

\newcommand{\cM}{\mathcal{M}}

\newcommand{\cO}{\mathcal{O}}

\newcommand{\cS}{\mathcal{S}}

\newcommand{\cV}{\mathcal{V}}

\newcommand{\sT}{\mathscr{T}}

\DeclareMathOperator{\Hom}{Hom}

\DeclareMathOperator{\Spec}{Spec}

\newcommand{\id}{\mathrm{id}}

\newcommand{\SH}{\mathrm{SH}}
\newcommand{\DM}{\mathrm{DM}}
\newcommand{\DA}{\mathrm{DA}}

\newcommand{\Sh}{\mathrm{Sh}}

\newcommand{\Sp}{\mathrm{Sp}}

\newcommand{\cofib}{\mathrm{cofib}}

\newcommand{\infCat}{\mathrm{Cat}}
\newcommand{\CAlg}{\mathrm{CAlg}}
\newcommand{\Fun}{\mathrm{Fun}}

\newcommand{\PrL}{\mathrm{Pr}^\mathrm{L}}

\newcommand{\Corr}{\mathrm{Corr}}
\newcommand{\cart}{\mathrm{cart}}
\newcommand{\op}{\mathrm{op}}

\newcommand{\unit}{\mathbf{1}}

\newcommand{\dNis}{\mathrm{dNis}}
\newcommand{\setale}{\mathrm{s\acute{e}t}}
\newcommand{\ketale}{\mathrm{k\acute{e}t}}
\newcommand{\ksm}{\mathrm{ksm}}
\newcommand{\kSm}{\mathrm{kSm}}

\newcommand{\Set}{\mathrm{Set}}
\newcommand{\Sch}{\mathrm{Sch}}
\newcommand{\lSch}{\mathrm{lSch}}
\newcommand{\lSpc}{\mathrm{lSpc}}

\newcommand{\Sm}{\mathrm{Sm}}

\newcommand{\rank}{\mathrm{rank}}
\newcommand{\ver}{\mathrm{ver}}
\newcommand{\divi}{\mathrm{div}}

\newcommand{\ol}{\overline}
\newcommand{\pt}{\mathrm{pt}}

\newcommand{\eSm}{\mathrm{eSm}}
\newcommand{\sSm}{\mathrm{sSm}}
\newcommand{\lSm}{\mathrm{lSm}}

\newcommand{\gp}{\mathrm{gp}}

\newcommand{\Spc}{\mathrm{Spc}}

\newcommand{\colim}{\mathop{\mathrm{colim}}}

\newcommand{\ul}{\underline}
\newcommand{\Ho}{\mathrm{Ho}}

\newcommand{\Mod}{\mathrm{Mod}}
\newcommand{\Op}{\mathrm{Op}}
\newcommand{\Fin}{\mathrm{Fin}}

\newcommand{\LMod}{\mathrm{LMod}}

\newcommand{\ML}{\mathrm{M\Lambda}}

\newcommand{\st}{\mathrm{st}}
\newcommand{\ex}{\mathrm{ex}}

\begin{document}
\title{Log motivic exceptional direct image functors}
\author{Doosung Park}
\address{Department of Mathematics and Informatics, University of Wuppertal, Germany}
\email{dpark@uni-wuppertal.de}

\begin{abstract}
In this paper, we construct the motivic exceptional direct image functors for fs log schemes. This construction is a part of the motivic six-functor formalism for fs log schemes.
\end{abstract}
\maketitle
\section{Introduction}

The Grothendieck six-functor formalism of schemes consists of the functors $f^*$, $f_*$, $f_!$, $f^!$, $\otimes$, and $Hom$ (internal Hom) together with the relations among them.
Grothendieck and his collaborators \cite{SGA4} show that the derived categories of $\ell$-adic sheaves satisfy this formalism.
Ayoub \cite{Ayo071},\cite{Ayo072} establishes this formalism in the motivic setting,
and he defines $f_!$ for quasi-projective morphisms.

\

In both the $\ell$-adic and motivic settings,
one can construct $f^*$, $f_*$, $\otimes$, and $Hom$ by formal categorical methods.
On the other hand,
defining $f_!$ requires an extra effort.
A typical restriction on $f$ is that $f$ is a compactifiable morphism,
which means that $f$ admits a factorization $X\xrightarrow{j} Y\xrightarrow{g} S$ such that $j$ is an open immersion and $g$ is proper.
Then one can define $f_!:=p_*j_!$,
and one needs to show that this $f_!$ is independent of the compactifications.
After that, $f^!$ is defined to be a right adjoint of $f_!$.
This is Deligne's argument in \cite{SGA4}.
Cisinski and D\'eglise \cite{CD19} use this argument to construct $f_!$ in the motivic setting for compactifiable morphisms.

\

This paper is the continuation of the series developed in \cite{logA1}, \cite{divspc}, and \cite{logGysin}, which explores $\A^1$-homotopy theory of fs log schemes.
The purpose of this paper is to construct the exceptional direct image functor $f_!$ in the $\A^1$-invariant log motivic setting.
We plan to study further properties of $f_!$ including the Poincar\'e duality in \cite{logsix}.

\

Let $f\colon X\to S$ be a morphism of fs log schemes such that its underlying morphism of schemes $\ul{f}\colon \ul{X}\to \ul{S}$ is compactifiable.
Then the morphism $f$ admits the induced factorization
\[
X\xrightarrow{q} \ul{X}\times_{\ul{S}}S\xrightarrow{p} S,
\]
where $p$ is the projection.
One can define $f_!:=q_*p_!$,
where $p_!$ is further defined using a compactification of $\ul{f}\colon \ul{X}\to \ul{S}$.
This is Nakayama's argument in \cite[\S 5.4]{MR1457738}.
Since $q$ is canonical,
we only need to show that $p_!$ is independent of the compactifications to achieve the main goal.

\subsection{Log motivic \texorpdfstring{$\infty$-}{infinity }categories}

To provide an axiomatic argument,
we introduce the notion of a \emph{log motivic $\infty$-category},
whose terminology is imitating a motivic triangulated category due to Cisinski-D\'eglise \cite[Definition 2.4.45]{CD19}.
A log motivic $\infty$-category $\sT$ is a dividing Nisnevich sheaf of presentably symmetric monoidal $\infty$-categories
\[
\sT
\in
\Sh_{\dNis}(\lSch/B,\CAlg(\PrL))
\]
satisfying certain conditions,
where $B$ is a finite dimensional noetherian base scheme throughout the paper.
See Definition \ref{logmotivic.1} for the details.

\

The fundamental $\infty$-category in $\A^1$-homotopy theory of fs log schemes is  $\SH$ in \cite[Definition 2.29]{logA1},
which is an extension of the original $\SH$ due to Morel-Voevodsky in \cite{MV} from schemes to fs log schemes.

\

The theory of log transfers over a perfect field is introduced in \cite{logDM},
but we do not know the theory of log transfers over arbitrary log schemes yet.
Hence we do not know how to extend the $\infty$-category of Voevodsky's motives $\DM$ from schemes to arbitrary log schemes.

\

Instead, we consider the functor $\Mod_\E\colon \lSch/B\to \CAlg(\PrL)$ in Construction \ref{logChow.9} for a commutative algebra object $\E$ of $\SH(B)$ such that $\Mod_\E(X)$ is the symmetric monoidal $\infty$-category of $\E$-module objects of $\SH(X)$ for $X\in \lSch/B$.
Let $\ML$ denotes the motivic Elienberg-MacLane spectrum in $\SH(B)$ for a commutative ring $\Lambda$,
which is introduced by Voevodsky \cite{zbMATH01194164},
but we use here Spitzweck's version \cite{MR3865569}.
If $k$ is a field and the exponential characteristic $p$ of $k$ is invertible in $\Lambda$,
then Cisinski and D\'eglise \cite[Theorem 3.1]{MR3404379} show that the canonical functor
\[
\Mod_{\ML}(X)
\to
\DM(X,\Lambda)
\]
is an equivalence of triangulated categories after taking homotopy categories (and hence an equivalence of $\infty$-categories) for smooth schemes $X$ over $k$.
In the case that $k$ has characteristic $0$ and $X=\Spec(k)$,
this result is due to R\"{o}ndig-{\O}stv{\ae}r \cite{MR2435654}.
Hence our $\Mod_{\ML}(X)$ for log smooth fs log schemes $X$ over $k$ can be considered as a logarithmic generalization of $\DM(X,\Lambda)$ in the case that $p$ is invertible in $\Lambda$.
We show that these examples are log motivic $\infty$-categories as follows:

\begin{thm}[Theorems \ref{restriction.15} and \ref{logChow.8}]
let $\E$ be a commutative algebra object of $\SH(B)$.
Then $\SH$ and $\Mod_{\E}$ are log motivic $\infty$-categories.
\end{thm}

Nakayama and Ogus \cite[Theorem 0.3]{zbMATH05809283} show that the Kato-Nakayama realization of an exact log smooth morphism is a topological submersion.
This indicates that when $\sT$ is a log motivic $\infty$-category,
it is easier to study $\sT^\ex$ in Definition \ref{restriction.1} for the six-functor formalism.
This is the full subcategory generated by the motives associated with exact log smooth schemes under colimits, shifts, and Tate twists.
In this paper and the sequel \cite{logsix}, we are mainly interested in the properties of $\sT^\ex$.

\subsection{Support property}

Cisinski and D\'eglise \cite[Definition 2.2.5]{CD19} introduce the support property in the motivic setting,
which implies that the construction of $f_!$ is independent of the choice of the compactifications.
We show a similar property in the log motivic setting as follows:

\begin{thm}[Theorem \ref{nonversupp.13}]
Let $\sT$ be a log motivic $\infty$-category.
Then $\sT^\ex$ satisfies the support property in the following sense.
Let 
\[
\begin{tikzcd}
V\ar[r,"j'"]\ar[d,"g"']&
X\ar[d,"f"]
\\
U\ar[r,"j"]&
S
\end{tikzcd}
\]
be a cartesian square in $\lSch/B$ such that $f$ is proper and $j$ is an open immersion.
Then the canonical natural transformation $Ex\colon j_\sharp g_*\to f_*j_\sharp'$ is an isomorphism.
\end{thm}

For the organization of the proof,
see the beginning of \S \ref{supp}.

\

Let $(\lSch/B)_{\mathrm{comp}}$ denote the subcategory of $\lSch/B$ spanned by compactifiable morphisms.
Using the support property,
we prove the following in Theorem \ref{base.2}:

\begin{thm}
There exists a functor
\[
\sT_!^\ex
\colon
(\lSch/B)_{\mathrm{comp}}
\to
\infCat_\infty
\]
such that $\sT_!^\ex(f)\simeq f_*$ (resp.\ $\sT_!^\ex(f)\simeq f_\sharp$) if $f$ is a proper morphism (resp.\ an open immersion).
\end{thm}

To show this, we need the technique of Liu-Zheng in \cite{LZ} for the $\infty$-categorical independence of the choice of the compactifications.
We also show a base change property under a certain assumption on exactness,
which is contained in Theorem \ref{base.2} too.

\subsection*{Acknowledgements}

This research was conducted in the framework of the DFG-funded research training group GRK 2240: \emph{Algebro-Geometric Methods in Algebra, Arithmetic and Topology}.

\subsection*{Notation and conventions}

Our standard reference for log geometry is Ogus's book \cite{Ogu}.
We employ the following notation throughout this article:

\begin{tabular}{l|l}
$B$ & base noetherian scheme of finite dimension
\\
$\Sch/B$ & category of schemes of finite type over $B$
\\
$\lSch/B$ &  category of fs log schemes of finite type over $B$
\\
$\lSpc/B$ & category of divided log spaces over $B$
\\
$\Sm$ & class of smooth morphisms in $\Sch/B$
\\
$\lSm$ & class of log smooth morphisms in $\lSch/B$
\\
$\eSm$ & class of exact log smooth morphisms in $\lSch/B$
\\
$\sSm$ & class of strict smooth morphisms in $\lSch/B$
\\
$\Hom_{\cC}$ & Hom space in an $\infty$-category $\cC$
\\
$\Spc$ & $\infty$-category of spaces
\\
$\Sh_t(\cC,\cV)$ & $\infty$-category of $t$-sheaves with values in an $\infty$-category $\cV$ on $\cC$
\\
$\id\xrightarrow{ad}f_*f^*$ & unit of an adjunction $(f^*,f_*)$
\\
$f^*f_*\xrightarrow{ad'}\id$ & counit of an adjunction $(f^*,f_*)$
\end{tabular}

\section{Exact log smooth motives}

We introduce the notion of a log motivic $\infty$-category $\sT$ in \S \ref{logmotivic}.
We show in \S \ref{module} that $\SH$ and $\Mod_\E$ are log motivic $\infty$-categories for every commutative algebra object $\E$ of $\SH(B)$ in \S \ref{module}.
We obtain $\sT^\ex$ from $\sT$ in \S \ref{restriction} by restricting to exact log smooth motives.
This paper and its sequel \cite{logsix} are mainly concerned about $\sT^\ex$.
In \S \ref{properties},
we explore several basic properties of $\sT^\ex$.
There is a further restriction $\sT^\st$ to strict smooth motives in \S \ref{generation},
which has a technical advantage as shown in Proposition \ref{generation.4}.
In \S \ref{ksm}, we discuss Kummer smooth motives under strict \'etale descent assumption.

\subsection{Log motivic \texorpdfstring{$\infty$-}{infinity }categories}
\label{logmotivic}

\begin{df}
\label{logmotivic.1}
A \emph{log motivic $\infty$-category} is a dividing Nisnevich sheaf of presentably symmetric monoidal $\infty$-categories
\[
\sT
\in
\Sh_{\dNis}(\lSch/B,\CAlg(\PrL))
\]
satisfying the following conditions:
\begin{itemize}
\item For every morphism $f$ in $\lSch/B$, let $f_*$ denote a right adjoint of $f^*:=\sT(f)$.
\item For every log smooth morphism $f\colon X\to S$ in $\lSch/B$, $f^*$ admits a left adjoint $f_\sharp$.
We set
\[
M_S(X)
:=
f_\sharp \unit_X,
\]
where $\unit_X$ (or simply $\unit$) is the monoidal unit of $\sT(X)$.
\item ($\lSm$-BC) For every cartesian square
\[
\begin{tikzcd}
X'\ar[d,"f'"']\ar[r,"g'"]&
X\ar[d,"f"]
\\
S'\ar[r,"g"]&
S
\end{tikzcd}
\]
such that $f$ is log smooth, the induced natural transformation
\[
Ex\colon f_\sharp'g'^*
\to
g^*f_\sharp
\]
is an isomorphism.
\item ($\lSm$-PF) For every log smooth morphism $f\colon X\to S$,
the induced natural transformation
\[
Ex\colon f_\sharp ((-)\otimes f^*(-))
\to
f_\sharp (-)\otimes (-)
\]
is an isomorphism.
\item (Loc) Let $i$ be a closed immersion in $\lSch/B$ with its open complement $j$.
Then the pair $(i^*,j^*)$ is conservative, and $i_*$ is fully faithful.
\item ($\A^1$-inv)
Let $p\colon X\times \A^1\to X$ be the projection, where $X\in \lSm/S$ and $S\in \lSch/B$.
Then $M_S(p)$ is an isomorphism.
\item ($\divi$-inv)
Let $f\colon Y\to X$ be a dividing cover in $\lSm/S$, where $S\in \lSch/B$.
Then $M_S(f)$ is an isomorphism.
\item ($\ver$-inv)
Let $j\colon X-\partial_S X\to X$ be the open immersion, where $X\in \lSm/S$ and $S\in \lSch/B$.
Then $M_S(j)$ is an isomorphism.
We refer to \cite[Definition 2.13]{logA1} for the notation $\partial_S X$,
which is the vertical boundary of $X$ over $S$.
If $S$ has the trivial log structure,
then $\partial_S X$ is the boundary $\partial X$ of $X$ consisting of the points $x\in X$ such that $\ol{\cM}_{X,x}$ is nontrivial.
\item ($\P^1$-Stab)
For $S\in \lSch/B$,
the object
\[
\unit(1)
:=
\cofib(M_S(S)\xrightarrow{M_S(i_0)}M_S(\P_S^1))[-2]
\]
of $\sT(S)$ is $\otimes$-invertible,
where $i_0\colon S\to \P_S^1$ is the zero section.
\item For $S\in \lSch/B$,
the family
\[
\{M_S(X)(d)[n]:X\in \lSm/S,\;d,n\in \Z\}
\]
generates $\sT(S)$ under colimits.
Furthermore,
there exists a regular cardinal $\kappa$ such that $M_S(X)(d)[n]$ is $\kappa$-compact \cite[Definition 5.3.4.5]{HTT} in $\sT(S)$ for every $X\in \lSm/S$ and $d,n\in \Z$.
\item Let $f\colon X\to S$ be a log smooth morphism in $\lSch/B$ with $S\in \Sch/B$.
Consider the induced cartesian square
\begin{equation}
\label{outline.0.1}
\begin{tikzcd}
\partial X\ar[r,"i'"]\ar[d,"p'"']&
X\ar[d,"p"]
\\
\ul{\partial X}\ar[r,"i"]&
\ul{X}
\end{tikzcd}
\end{equation}
in $\lSch/B$.
Then the square
\begin{equation}
\label{outline.0.2}
\begin{tikzcd}
f^*\ar[d,"ad"']\ar[r,"ad"]&
i_*i^*f^*\ar[d,"ad"]
\\
p_*p^*f^*\ar[r,"ad"]&
q_*q^*f^*
\end{tikzcd}
\end{equation}
is cartesian, where $q=pi'=ip'$.
\item For $X\in \lSch/B$ such that $X$ admits a chart $\N$,
consider the induced cartesian square of the form \eqref{outline.0.1}.
Then the square
\begin{equation}
\label{outline.0.3}
\begin{tikzcd}
\id\ar[d,"ad"']\ar[r,"ad"]&
i_*i^*\ar[d,"ad"]
\\
p_*p^*\ar[r,"ad"]&
q_*q^*
\end{tikzcd}
\end{equation}
is cartesian, where $q=pi'=ip'$.
\end{itemize}
A log motivic $\infty$-category $\sT$ is called \emph{compactly generated} if $M_S(X)$ is a compact object of $\sT(S)$ for every $S\in \lSch/B$ and $X\in \lSm/S$.
\end{df}

A notable difference with a motivic triangulated category of Cisinski-D\'eglise \cite[Definition 2.4.45]{CD19} is that we need the axioms ($\divi$-inv) and ($\ver$-inv) to control the boundary behavior of fs log schemes.
We do not know whether the last two axioms of a log motivic $\infty$-category are consequences of the other axioms or not,
so we need to include them unfortunately.
Nevertheless,
we manage to show that $\SH$ is a log motivic $\infty$-category as follows.

\begin{thm}
\label{restriction.15}
The functors
\[
\SH,\DA(-,\Lambda),\DA_{\setale}(-,\Lambda),\DA_{\ketale}(-,\Lambda)\colon (\lSch/B)^{\op}\to \CAlg(\PrL)
\]
obtained from \cite[Definition 2.29]{logA1} are log motivic $\infty$-categories,
where $\Lambda$ are commutative ring.
\end{thm}
\begin{proof}
We focus on the case of $\SH$ since the proofs are similar.
If $S\in \lSch/B$, then we have
\[
\SH(S)
\simeq
\Sp_{\P^1}((\A^1\cup \ver)^{-1}\Sh_{\dNis}(\lSm/S,\Spc)),
\]
which is the $\P^1$-stablization of the $\infty$-category of $(\A^1\cup \ver)$-local dividing Nisnevich sheaves of spaces on $\lSm/S$.
By a standard argument,
$\Sigma^{n,d}\Sigma_{\P^1}^\infty X_+$ is a compact object of $\SH(S)$ for all $X\in \lSm/S$ and integers $d$ and $n$.
We see that $\SH$ is a log motivic $\infty$-category together with \cite[Theorems 4.31, 4.32]{logA1}, \cite[Example 2.0.2]{logGysin}, and Theorem \ref{update.2}.

\end{proof}

\subsection{Modules}
\label{module}

\begin{const}
\label{logChow.4}
Here, we review the theory of modules in higher algebra.
Let $\cC^\otimes$ be a symmetric monoidal $\infty$-category.
In \cite[Definition 3.3.3.8]{HA},
we have the $\infty$-operad $\Mod(\cC)^{\otimes}$ of module objects of $\cC$.
Let $\Mod(\cC)$ denote its underlying $\infty$-category.
On the other hand,
in \cite[Definitions 4.2.1.13, 4.2.2.10]{HA},
we have the $\infty$-category $\LMod(\cC)$ of left module objects of $\cC$ and $\infty$-category $\LMod^{\A_\infty}(\cC)$ of $\A_\infty$-module objects of $\cC$.
For a commutative algebra object $A$ in $\cC$,
we have the full subcategories $\Mod_A(\cC)^\otimes$, $\LMod_A^{\A_\infty}(\cC)$, and $\LMod_A(\cC)$.

Let $\mathbf{\Delta}$ denote the category of combinatorial simplices.
We have the $\infty$-operads $\mathcal{LM}^\otimes$ in \cite[Definition 4.2.1.7]{HA},
whose underlying $\infty$-category is a disjoint union of the two vertices $\mathbf{a}$ and $\mathbf{m}$,
see \cite[Remark 4.2.1.8]{HA}.
There exists a natural commutative diagram
\[
\begin{tikzcd}
\Mod(\cC)\ar[r]\ar[d,hookrightarrow]&
\LMod^{\A_\infty}(\cC)\ar[d,hookrightarrow]\ar[r,leftarrow]&
\LMod(\cC)\ar[d,hookrightarrow]
\\
\Fun_{\Fin_*}((\Fin_*)_{\langle 1\rangle/},\cC^\otimes)
\ar[r,"\Phi^*"]&
\Fun_{\Fin_*}(\mathbf{\Delta}^{\op}\times \Delta^1,\cC^\otimes)\ar[r,leftarrow,"\gamma^*"]&
\Fun_{\Fin_*}(\mathcal{LM}^{\otimes},\cC^\otimes)
\end{tikzcd}
\]
whose vertical arrows are fully faithful,
see \cite[Proposition 4.2.2.12]{HA} for the right square and the proof of \cite[Proposition 4.5.1.4]{HA} for the left square.
To obtain the forgetful functors from $\Mod(\cC)$, $\LMod^{\A_\infty}(\cC)$, and $\LMod(\cC)$ to $\cC$,
we restrict $(\Fin_*)_{\langle 1\rangle /}$ (resp.\ $\mathbf{\Delta}^{\op}\times \Delta^1$, resp.\ $\mathcal{LM}^{\otimes}$) to the object $\id\colon \langle 1\rangle \to \langle 1\rangle$ (resp.\ $([0],0)$, resp.\ $\mathbf{m}$) and $\cC^\otimes$ to $\cC:=\cC^\otimes_{\langle 1\rangle}$.
Hence we obtain a commutative diagram
\begin{equation}
\label{logChow.4.1}
\begin{tikzcd}
\Mod(\cC)\ar[rd,"U"']\ar[r]&
\LMod^{\A_\infty}(\cC)\ar[r,leftarrow]\ar[d,"U"]&
\LMod(\cC)\ar[ld,"U"]
\\
&
\cC
\end{tikzcd}
\end{equation}
whose vertical and diagonal arrows are the forgetful functors.
Note that the functor $\LMod(\cC)\to \LMod^{\A_\infty}(\cC)$ is an equivalence of $\infty$-categories by \cite[Proposition 4.2.2.12]{HA}.
We have a similar commutative diagram for $\Mod_A(\cC)$, $\LMod_A^{\A_\infty}(\cC)$, and $\LMod_A(\cC)$.
By \cite[Corollary 4.5.1.5]{HA},
the functor $\Mod_A(\cC)\to \LMod_A^{\A_\infty}(\cC)$ is an equivalence of $\infty$-categories.
\end{const}

\begin{const}
\label{logChow.5}
Let $\Op_\infty$ denote the $\infty$-category of $\infty$-operads,
which contains $\CAlg(\infCat_\infty)$ as a subcategory according to \cite[Variant 2.1.4.13]{HA}.
For the functoriality of the module construction,
consider a functor of $\infty$-categories $\cE^{\otimes}\colon \cS\to \Op_\infty$,
where $\cS$ is an $\infty$-category.
Due to \cite[Remark 5.3.1.20]{HA},
this corresponds a cocartesian $\cS$-family of $\infty$-operads $\cE^{\otimes}\to \Fin_*\times \cS$ in the sense of \cite[Definition 5.3.1.19]{HA}.
Let $R$ be a cocartesian $\cS$-family of commutative algebra (i.e., $\Fin_*$-algebra) objects of $\cE^{\otimes}$ in the sense of \cite[Remark 5.3.1.22]{HA}.
By \cite[Definition 5.3.1.23, Remark 5.3.1.24]{HA},
there exists an $\infty$-category $\Mod_R^{\cS}(\cE)^{\otimes}$ equipped with a forgetful functor $\Mod_R^{\cS}(\cE)^{\otimes}\to \Fin_*\times \cS$ such that there exists a natural equivalence
\[
\Mod_R^{\cS}(\cE)^{\otimes} \times_{\cS} \{s\}
\simeq
\Mod_{R_s}(\cE_s)^{\otimes}
\]
for $s\in \cS$.
Furthermore,
\cite[Proposition 5.3.1.25(1)]{HA} implies that the forgetful functor $\Mod_R^{\cS}(\cE)^{\otimes}\to \Fin_*\times \cS$ is a cocartesian $\cS$-family of $\infty$-operads.
This corresponds to a functor
\[
\Mod_R^{\cS}(\cE)
\colon
\cS\to \Op_\infty
\]
due to \cite[Remark 5.3.1.20]{HA}.
In particular,
we obtain a natural map of $\infty$-operads
\[
\Mod_{R_s}(\cE_s)^{\otimes}
\to
\Mod_{R_t}(\cE_t)^{\otimes}
\]
for every morphism $s\to t$ in $\sT$.

Let $\cK$ denote the full subcategory of $\Fun(\Delta^1,\Fin_*)$ spanned by semi-inert morphisms in $\Fin_*$ in the sense of \cite[Notation 3.3.2.1]{HA}.
We have the functor $\Fin_*\to \cK$ sending $\langle n \rangle$ to $\id\colon \langle n \rangle \to \langle n\rangle$ for every integer $n\geq 0$.
Using this functor,
we obtain the map of simplicial sets
\[
\Fun_{\Fun(\{1\},\Fin_*)}
(X\times_{\Fun(\{0\},\Fin_*)}\cK,\cE^\otimes)
\to
\Fun_{\Fin_*}
(X,\cE^\otimes)
\]
for every simplicial set $X$ over $\Fin_*$.
Together with \cite[Construction 3.3.3.1, Definition 3.3.3.8]{HA},
we obtain the forgetful functor $\Mod(\cE)^\otimes \to \cE^\otimes$.
Using \cite[Remark 5.3.1.20]{HA},
we obtain a commutative square of $\infty$-categories
\begin{equation}
\label{logChow.4.2}
\begin{tikzcd}
\Mod_{R_s}(\cE_s)^\otimes
\ar[d,"U"']\ar[r]&
\Mod_{R_t}(\cE_t)^\otimes
\ar[d,"U"]
\\
\cE_s^\otimes
\ar[r]&
\cE_t^\otimes
\end{tikzcd}
\end{equation}
whose vertical arrows are the forgetful functors for every morphism $s\to t$ in $\cS$.
\end{const}

\begin{const}
\label{logChow.7}
Suppose that $\cC^\otimes$ is a presentably symmetric monoidal $\infty$-category.
Let $A$ be a commutative algebra object of $\cC$.
As a consequence of \cite[Theorem 4.5.2.1]{HA},
$\Mod_A(\cC)^\otimes$ is a presentably symmetric monoidal $\infty$-category.
Furthermore, the commutativity of \eqref{logChow.4.1} and \cite[Corollaries 4.2.3.2, 4.2.3.7(1),(2)]{HA} imply that 
$\Mod_A(\cC)$ is presentable and the forgetful functor
\[
U
\colon
\Mod_A(\cC)
\to
\cC
\]
is conservative and preserves colimits and limits.
This admits a left adjoint
\[
(-)\otimes A \colon \cC\to \Mod_A(\cC)
\]
by \cite[Corollary 4.2.4.8]{HA},
i.e., the composite $\cC\xrightarrow{(-)\otimes A} \Mod_A(\cC)\xrightarrow{U} \cC$ sends $X\in \cC$ to $X\otimes A$.
\end{const}

\begin{lem}
\label{logChow.6}
Let $F\colon \cC^\otimes \to \cD^\otimes$ be a symmetric monoidal functor of symmetric monoidal presentable $\infty$-categories.
If $A$ is a commutative algebra object of $\cC$,
then the induced functor
\[
G\colon \Mod_{A}(\cC)^\otimes
\to
\Mod_{F(A)}(\cD)^\otimes
\]
is symmetric monoidal.
\end{lem}
\begin{proof}
Observe from Construction \ref{logChow.5} that $G$ is a map of $\infty$-operads, i.e., a lax symmetric monoidal functor.
Since $G$ sends the unit $A$ in $\Mod_A(\cC)$ to the unit $F(A)$ in $\Mod_{F(A)}(\cD)$,
it suffices to show that the induced morphism
\[
G(X)\otimes G(Y)
\to
G(X\otimes Y)
\]
is an isomorphism for $X,Y\in \Mod_A(\cC)$.
The forgetful functor $U\colon \Mod_{F(A)}(\cD)\to \cD$ is conservative,
so it suffices to show this after applying $U$ in the front.
By the description of the monoidal structure on $\Mod_A(\cC)$ in \cite[\S 4.4.1, Theorem 4.5.2.1(2)]{HA} and the commutativity of \eqref{logChow.4.2},
it suffices to show 
\begin{align*}
&
\colim
\big(
\cdots
\;
\substack{\rightarrow\\[-1em] \rightarrow \\[-1em] \rightarrow}
\;
FU(X)\otimes F(A) \otimes FU(Y)
\;
\substack{\rightarrow\\[-1em] \rightarrow}
\;
FU(X)\otimes FU(Y)
\big)
\\
\simeq
&
F\big(
\colim
\big(
\cdots
\;
\substack{\rightarrow\\[-1em] \rightarrow \\[-1em] \rightarrow}
\;
U(X)\otimes A \otimes U(Y)
\;
\substack{\rightarrow\\[-1em] \rightarrow}
\;
U(X)\otimes U(Y)
\big)
\big),
\end{align*}
where the colimit diagrams are obtained by the two-sided bar constructions.
This is clear due to the monoidality of $F$.
\end{proof}

\begin{const}
\label{logChow.9}
Let $\E$ be a commutative algebra object of $\sT(B)$,
where $\sT$ is a log motivic $\infty$-category.
Consider the induced functor $\sT\colon (\lSch/B)^{\op}\to \Op_\infty$.
The collection of $p^*\E$ for all morphisms $p\colon S\to B$ in $\lSch/B$ yields a natural transformation $\E\colon *\to \sT\colon (\lSch/B)^{\op}\to \Op_\infty$,
where $*$ is the constant functor to $\Fin_*$.
By \cite[Remark 5.3.1.20]{HA},
this corresponds to a cocartesian $(\lSch/B)^{\op}$-family of commutative algebra objects of $\sT$.
We set
\[
\Mod_{\E}^{\otimes}
:=
\Mod_{\E}^{(\lSch/B)^{\op}}(\sT)^{\otimes}
\colon
(\lSch/B)^{\op}
\to
\Op_\infty.
\]
using Construction \ref{logChow.5}.
Observe that for every morphism $p\colon S\to B$ in $\lSch/B$,
we have
\[
\Mod_{\E}(S)^{\otimes}
\simeq
\Mod_{p^*\E}(\sT(S))^{\otimes}.
\]
For every morphism $f\colon X\to S$ in $\lSch/B$,
the induced functor
\[
f^*
\colon
\Mod_\E(S)
\to
\Mod_\E(X)
\]
is a colimit preserving of presentable $\infty$-categories by the commutativity of \eqref{logChow.4.2} and Construction \ref{logChow.7}.
Furthermore,
the induced functor
\[
f^*
\colon
\Mod_\E(S)^\otimes
\to
\Mod_\E(X)^\otimes
\]
is a symmetric monoidal functor of symmetric monoidal $\infty$-categories by Construction \ref{logChow.7} and Lemma \ref{logChow.6}.
Hence we obtain the functor
\[
\Mod_{\E}^\otimes
\colon
(\lSch/B)^{\op}
\to
\CAlg(\PrL)
\]
by restricting the codomain.
We often use the notation $\Mod_{\E}$ instead from now on when no confusion seems likely to arise.
\end{const}

\begin{thm}
\label{logChow.8}
Let $\E$ be a commutative algebra object of $\sT(B)$,
where $\sT$ is a compactly generated log motivic $\infty$-category.
Then $\Mod_\E$ is a compactly generated log motivic $\infty$-category.
\end{thm}
\begin{proof}
For $S\in \lSch/B$,
we have the adjunction
\[
-\otimes \E
:
\sT(S)
\leftrightarrows
\Mod_{\E}(S)
:
U,
\]
where $U$ is the forgetful functor.
For $V\in \lSm/S$ and integers $d$ and $n$,
we set
\[
M_S(V)(d)[n]:=
\Sigma^{n,d}\Sigma^\infty V_+ \otimes \E
\in
\Mod_{\E}(S).
\]
Using the conservativity of the forgetful functor $U$,
we see that the family
\begin{equation}
\label{logChow.8.5}
\{M_S(V)(d)[n]:V\in \lSm/S,\;d,n\in \Z\}
\end{equation}
generates $\Mod_\E(S)$ and every object of the family is compact.

For $S\in \lSch/B$,
the object 
$\unit(-1):= \Sigma^{0,-1} \unit \otimes \E$ of $\Mod_{\E}(S)$ is a $\otimes$-inverse of $\unit(1)\simeq \Sigma^{0,1}\unit \otimes \E$.
It follows that $\Mod_{\E}$ satisfies ($\P^1$-stab).

From \eqref{logChow.4.2},
we have a commutative square
\begin{equation}
\label{logChow.8.1}
\begin{tikzcd}
\Mod_{\E}(S)\ar[d,"U"']\ar[r,"f^*"]&
\Mod_{\E}(X)\ar[d,"U"]
\\
\sT(S)\ar[r,"f^*"]&
\sT(X)
\end{tikzcd}
\end{equation}
for every morphism $f\colon X\to S$ in $\lSch/B$.
Furthermore,
we have a commutative square
\begin{equation}
\label{logChow.8.2}
\begin{tikzcd}
\sT(S)\ar[r,"f^*"]\ar[d,"(-)\otimes \E_S"']&
\sT(X)\ar[d,"(-)\otimes \E_X"]
\\
\Mod_{\E}(S)\ar[r,"f^*"]&
\Mod_{\E}(X),
\end{tikzcd}
\end{equation}
whose commutativity can be checked by composing with the conservative forgetful functor $U\colon \Mod_{\E}(X)\to \sT(X)$.
By taking a right adjoint,
we have a commutative square
\begin{equation}
\label{logChow.8.3}
\begin{tikzcd}
\Mod_{\E}(X)\ar[d,"U"']\ar[r,"f_*"]&
\Mod_{\E}(S)\ar[d,"U"]
\\
\sT(X)\ar[r,"f_*"]&
\sT(S).
\end{tikzcd}
\end{equation}

If $f\colon X\to S$ is a morphism in $\lSm$,
then the commutativity of \eqref{logChow.8.1} implies that $f^*\colon \Mod_{\E}(S)\to \Mod_{\E}(X)$ admits a left adjoint $f_\sharp$ since the forgetful functors $U$ preserve colimits and are conservative due to Construction \eqref{logChow.7}.
By taking a left adjoint of \eqref{logChow.8.1},
we have a commutative square
\begin{equation}
\label{logChow.8.4}
\begin{tikzcd}
\sT(X)\ar[r,"f_\sharp"]\ar[d,"(-)\otimes \E_X"']&
\sT(S)\ar[d,"(-)\otimes \E_S"]
\\
\Mod_{\E}(X)\ar[r,"f_\sharp"]&
\Mod_{\E}(S).
\end{tikzcd}
\end{equation}

To show ($\lSm$-PF) for $\Mod_\E$,
it suffices to show that the natural morphism
\[
U f_\sharp (M_X(W) \otimes f^* M_S(V))
\xrightarrow{Ex}
U f_\sharp M_X(W)\otimes M_S(V)
\]
in $\sT(S)$ is an isomorphism since the family \eqref{logChow.8.5} generates $\Mod_{\E}(S)$.
This is a consequence of ($\lSm$-PF) for $\sT$ since \eqref{logChow.8.2} and \eqref{logChow.8.4} commute.

The remaining conditions of a log motivic $\infty$-category for $\Mod_{\E}$ are consequences of the corresponding conditions for $\sT$ and \cite[Lemma 2.0.1]{logGysin} since \eqref{logChow.8.1}, \eqref{logChow.8.2}, \eqref{logChow.8.3}, and \eqref{logChow.8.4} commute.
\end{proof}

\subsection{Restrictions to exact log smooth log schemes}
\label{restriction}

Throughout this subsection,
we fix a log motivic $\infty$-category $\sT$.

\begin{df}
\label{restriction.1}
For $S\in \lSch/B$, let $\sT^\ex(S)$ be the full subcategory generated under colimits, shifts, and Tate twists by $M_S(X)$ for all $X\in \eSm/S$.
\end{df}

\begin{prop}
\label{restriction.2}
For $S\in \lSch/B$, the family
\[
\cA:=\{M_S(X)(d)[n]:X\in \eSm/S,\; d,n\in \Z\}
\]
consists of $\kappa$-compact objects and generates $\sT^\ex(S)$ under colimits.
Furthermore, $\sT^\ex(S)$ is presentable.
\end{prop}
\begin{proof}
Assume that $\cF\in \sT^\ex(S)$ satisfies $\Hom_{\sT^\ex(S)}(M_S(X)(d)[n],\cF)\simeq 0$ for all $X\in \eSm/S$ and $d,n\in \Z$.
We need to show $\cF\simeq 0$.
Let $\cC$ be the full subcategory of $\sT^\ex(S)$ consisting of $\cG\in \sT^\ex(S)$ such that $\Hom_{\sT^\ex(S)}(\cG,\cF)\simeq 0$.
Then $\cC$ is closed under colimits and contains $M_S(X)(d)[n]$ for all $X\in \eSm/S$ and $d,n\in \N$.
It follows that $\cC$ contains $\sT^\ex(S)$.
This means $\Hom_{\sT^\ex(S)}(\cF,\cF)\simeq 0$, so $\cF\simeq 0$.
Hence $\cA$ generates $\sT^\ex(S)$ under colimits.

By our assumption on $\sT(S)$, $\cA$ consists of $\kappa$-compact objects.
Since $\sT(S)$ is presentable, $\sT^\ex(S)$ is locally small and admits sums.
Hence \cite[Corollary 1.4.4.2]{HA} shows that $\sT^\ex(S)$ is presentable too.
\end{proof}

\begin{const}
\label{restriction.3}
Let
\[
\varphi_\sharp\colon
\sT^\ex(S)
\to
\sT(S)
\]
be the inclusion functor.
Since $\varphi_\sharp$ preserves sums, $\varphi_\sharp$ preserves colimits by \cite[Proposition 1.4.4.1(2)]{HA}.
It follows that $\varphi_\sharp$ admits a right adjoint
\[
\varphi^*\colon \sT(S)\to \sT^\ex(S).
\]

For any morphism $f\colon X\to S$ in $\lSch/B$, we have a natural isomorphism
\[
\Hom_{\sT^\ex(S)}(-,\varphi^*f_*\varphi_\sharp (-))
\simeq
\Hom_{\sT(S)}(f^*\varphi_\sharp (-),\varphi_\sharp(-)).
\]
The class $\eSm$ is closed under pullbacks.
Hence the essential image of $f^*\varphi_\sharp$ is contained in the essential image of $\varphi_\sharp$, so we have natural isomorphisms
\begin{align*}
\Hom_{\sT(S)}(f^*\varphi_\sharp (-),\varphi_\sharp(-))
\simeq &
\Hom_{\sT^\ex(S)}(\varphi^*f^*\varphi_\sharp (-),\varphi^*\varphi_\sharp(-))
\\
\simeq &
\Hom_{\sT^\ex(S)}(\varphi^*f^*\varphi_\sharp (-),-).
\end{align*}
It follows that $f^*:=\varphi^*f^*\varphi_\sharp$ is left adjoint to $f_*:=\varphi^*f_*\varphi_\sharp$.
Furthermore, we have a natural isomorphism
\begin{equation}
\label{restriction.3.1}
\varphi_\sharp f^*
\simeq
f^*\varphi_\sharp.
\end{equation}
On the other hand, we have a natural transformation
\begin{equation}
\label{restriction.3.2}
Ex\colon \varphi_\sharp f_*
\to
f_*\varphi_\sharp
\end{equation}
given by
\[
\varphi_\sharp \varphi^* f_* \varphi_\sharp
\xrightarrow{ad'}
f_*\varphi_\sharp.
\]

The class $\eSm$ is closed under compositions.
Hence the essential image of $f_\sharp \varphi_\sharp$ is contained in the essential image of $\varphi_\sharp$.
Assume $f\in \eSm$.
We can similarly show that $f^*=\varphi^*f^*\varphi_\sharp$ is right adjoint to $f_\sharp:=\varphi^*f_\sharp\varphi_\sharp$.
Furthermore, we have a natural isomorphism
\begin{equation}
\label{restriction.3.3}
\varphi_\sharp f_\sharp
\simeq
f_\sharp\varphi_\sharp.
\end{equation}
\end{const}

\begin{prop}
\label{restriction.4}
There exists a functor
\[
\sT^\ex
\colon
(\lSch/B)^{\op}
\to
\CAlg(\PrL)
\]
such that $f^*:=\sT^\ex(f)$ for every morphism $f$ in $\lSch/B$ coincides with the above construction.
\end{prop}
\begin{proof}
For a category $\bC$, let $\bC^\amalg$ denote the category defined as follows:
\begin{enumerate}
\item[(i)] An object of $\bC^\amalg$ is an $n$-tuple $(X_1,\ldots,X_n)\in \bC^n$ with $n\geq 0$.
\item[(ii)] A morphism $f\colon (X_1,\ldots,X_m)\to (Y_1,\ldots,Y_n)$ is a function $\alpha\colon \{1,\ldots,m\}\to \{1,\ldots,n\}$ equipped with morphisms $X_i\to Y_{\alpha(i)}$ for all $i\in \{1,\ldots,m\}$.
\end{enumerate}
Apply \cite[Proposition 2.4.1.7, Theorem 2.4.1.18]{HA} to $\sT\colon (\lSch/B)^{\op}\to \CAlg(\infCat_\infty)$ to obtain a functor
\[
\sT_{\amalg}\colon ((\lSch/B)^{\op})^{\amalg} \to \infCat_\infty
\]
that is a lax cartesian structure in the sense of \cite[Definition 2.4.1.1]{HA}.
This means
\[
\sT_{\amalg}(X_1,\ldots,X_n)
\simeq
\sT(X_1)\times \cdots \times \sT(X_n)
\]
for all $X_1,\ldots,X_n\in \lSch/B$ and integer $n\geq 0$.
We set
\[
\sT_{\amalg}^\ex(X_1,\ldots,X_n)
:=
\sT^\ex(X_1)\times \cdots \times \sT^\ex(X_n),
\]
and we regard $\sT_{\amalg}^\ex(X_1,\ldots,X_n)$ as a full subcategory of $\sT_{\amalg}(X_1,\ldots,X_n)$.
This construction yields a functor
\[
\sT_{\amalg}^\ex\colon ((\lSch/B)^{\op})^{\amalg} \to \infCat_\infty
\]
and a natural transformation $\varphi_\sharp\colon \sT_{\amalg}^\ex\to \sT_{\amalg}$.
By construction, $\sT_{\amalg}^\ex$ is a lax cartesian structure.
Use \cite[Proposition 2.4.1.7, Theorem 2.4.1.18]{HA} again to obtain
\[
\sT^\ex\colon (\lSch/B)^{\op} \to \CAlg(\infCat_\infty)
\]
and a natural transformation $\varphi_\sharp\colon \sT^\ex\to \sT$.

For all $S\in \lSch/B$, Proposition \ref{restriction.2} shows that $\sT^\ex$ is presentable.
For all morphism $f\colon X\to S$ in $\lSch/B$, $f^*\colon \sT(S)\to \sT(X)$ preserves colimits by assumption.
Since $\varphi_\sharp$ preserves colimits and is fully faithful, $\varphi_\sharp f^*\simeq f^*\varphi_\sharp$ shows that $f^*\colon \sT^\ex(S)\to \sT^\ex(X)$ preserves colimits too.
It follows that $\sT^\ex$ factors through $\CAlg(\PrL)$.
\end{proof}

\begin{thm}
\label{restriction.16}
The functor
\[
\sT^\ex
\colon
(\lSch/B)^{\op}
\to
\CAlg(\PrL)
\]
satisfies the following properties:
\begin{itemize}
\item Let $f\colon X\to S$ be an exact log smooth morphism in $\lSch/B$.
Then $f^*$ admits a left adjoint $f_\sharp$.
\item \textup{($\eSm$-BC)}
Let
\[
\begin{tikzcd}
X'\ar[d,"f'"']\ar[r,"g'"]&
X\ar[d,"f"]
\\
S'\ar[r,"g"]&
S
\end{tikzcd}
\]
be a cartesian square in $\lSch/B$ such that $f$ is exact log smooth.
Then the natural transformation
\[
f_\sharp'g'^*
\xrightarrow{Ex}
g^*f_\sharp
\]
is an isomorphism.
\item \textup{($\eSm$-PF)}
Let $f$ be an exact log smooth morphism in $\lSch/B$.
Then the natural transformation
\[
f_\sharp((-)\otimes f^*(-))
\xrightarrow{Ex}
f_\sharp(-)\otimes (-)
\]
is an isomorphism.
\item \textup{($\A^1$-inv)}
Let $p\colon X\times \A^1\to X$ be the projection,
where $X\in \lSch/B$.
Then $p^*$ is fully faithful.
\item \textup{($\divi$-inv)}
Let $f\colon X\to S$ be a dividing cover in $\lSch/B$.
Then $f^*$ is fully faithful.
\item \textup{($\ver$-inv)}
Let $f\colon X\to S$ be an exact log smooth morphism in $\lSch/B$, and let $j\colon X-\partial_S X\to X$ be the open immersion.
Then the natural transformation
\[
f^*
\xrightarrow{ad}
j_*j^*f^*
\]
is an isomorphism.
\item \textup{($\square$-inv)}
Let $p\colon X\times \square\to X$ be the projection,
where $X\in \lSch/B$.
Then $p^*$ is fully faithful.
\item \textup{($\P^1$-stab)}
For $S\in \lSch/B$, $\unit_S(1)$ is $\otimes$-invertible.
\item \textup{(Loc)}
Let $i$ be a strict closed immersion in $\lSch/B$,
and let $j$ be its open complement.
Then the pair of functors $(i^*,j^*)$ is conservative,
and the functor $i_*$ is fully faithful.
\end{itemize}
\end{thm}
\begin{proof}
The first point is shown in Construction \ref{restriction.3}.

Combine ($\lSm$-BC) for $\sT$, \eqref{restriction.3.1}, and \eqref{restriction.3.3} to show ($\eSm$-BC).
One can similarly show ($\eSm$-PF), ($\A^1$-inv), and ($\ver$-inv).
We deduce ($\square$-inv) from ($\A^1$-inv) and ($\ver$-inv).

If $f\colon X\to S$ is a dividing cover in $\lSch/B$,
then $f^*\colon \sT(S)\to \sT(X)$ is an equivalence.
Combine with \eqref{restriction.3.1} to show ($\divi$-inv).

Hence it remains to show (Loc).
We have the natural isomorphisms
\[
i^*\varphi^*i_*\varphi_\sharp
\simeq
\varphi^*i^*i_* \varphi_\sharp
\simeq
\varphi^*\varphi_\sharp
\simeq
\id,
\]
where (Loc) for $\sT$ is needed for the third isomorphism.
This means that the functor $i_*=\varphi^*i_*\varphi_\sharp$ is fully faithful.
Since $\varphi_\sharp$ is fully faithful, it remains to check that the pair of functors $(\varphi_\sharp i^*,\varphi_\sharp j^*)$ is conservative.
We have $\varphi_\sharp i^* \simeq i^*\varphi_\sharp$ and $\varphi_\sharp j^* \simeq  j^*\varphi_\sharp$.
Hence (Loc) for $\sT$ finishes the proof.
\end{proof}

\subsection{Properties of exact log smooth motives}
\label{properties}

Throughout this subsection,
we fix a log motivic $\infty$-category $\sT$.
In this subsection,
we collect several properties for $\sT^\ex$ that are needed later.

\begin{prop}
\label{restriction.5}
Let $i\colon Z\to S$ be a strict closed immersion in $\lSch/B$.
Then the natural transformation
\[
Ex\colon \varphi_\sharp i_*
\to
i_* \varphi_\sharp
\]
is an isomorphism.
\end{prop}
\begin{proof}
Consider the commutative diagram
\[
\begin{tikzcd}[row sep=small]
\varphi_\sharp j_\sharp j^*\ar[dd,"\simeq"']\ar[r,"ad'"]&
\varphi_\sharp\ar[dd,"\id"]\ar[r,"ad"]&
\varphi_\sharp i_*i^*\ar[d,"Ex"]
\\
&
&
i_*\varphi_\sharp i^*\ar[d,"\simeq"]
\\
j_\sharp j^* \varphi_\sharp\ar[r,"ad'"]&
\varphi_\sharp \ar[r,"ad'"]&
i_*i^*\varphi_\sharp.
\end{tikzcd}
\]
By (Loc),
the top and bottom rows are cofiber sequences.
It follows that the upper right vertical arrow is an isomorphism.
To conclude, observe that $i^*$ is essentially surjective since $i_*$ is fully faithful.
\end{proof}

\begin{prop}
\label{restriction.6}
Let $P$ be a sharp fs monoid, and let $\theta\colon P\to Q$ be an injective, local, and locally exact homomorphism of fs monoids with a $\theta$-critical face $G$ of $Q$.
Then for every morphism $S\to \A_P\times B$ in $\lSch/B$, the natural transformation
\[
f_\sharp j_\sharp j^* f^*
\xrightarrow{ad'}
f_\sharp f^*
\]
is an isomorphism, where $f\colon S\times_{\A_P}\A_Q\to S$ is the projection, and $j\colon S\times_{\A_P}\A_{Q_G}\to S\times_{\A_P}\A_Q$ is the induced open immersion.
\end{prop}
\begin{proof}
This is an immediate consequence of \cite[Theorem 3.2.4]{logGysin}.
\end{proof}

A morphism $f\colon X\to S$ of fs log schemes is called \emph{$\Q$-integral} if for every point $x\in X$,
the induced homomorphism $\cM_{S,f(x)}\to \cM_{X,x}$ is $\Q$-integral or equivalently locally exact \cite[Theorem I.4.7.7]{Ogu}.
By \cite[Lemma 1.3]{MR2604925},
every exact log flat morphism (and hence exact log smooth morphism also) of fs log schemes is $\Q$-integral.

\begin{prop}
\label{restriction.7}
Let $f\colon X\to S$ be a $\Q$-integral morphism in $\lSch/B$.
Then there exists a Kummer log smooth morphism $g\colon S'\to S$ such that $u^*\colon \sT^\ex(Y)\to \sT^\ex(Y\times_S S')$ is conservative for every pullback $u\colon Y\times_S S'\to Y$ of $g$ and the projection $X\times_S S'\to S'$ is saturated.
\end{prop}
\begin{proof}
By \cite[Proposition 3.3.3]{logGysin}, there exists a Kummer log smooth morphism $g\colon S'\to S$ such that $u^*\colon \sT(Y)\to \sT(Y\times_S S')$ is conservative for every pullback $u$ of $g$ and the projection $X\times_S S'\to S'$ is saturated.

Assume that $\alpha$ is a morphism in $\sT^\ex(S)$ such that $u^*\alpha$ is an isomorphism.
Since $\varphi_\sharp u^*\simeq u^*\varphi_\sharp$, $u^*\varphi_\sharp \alpha$ is an isomorphism.
By the above paragraph, $\varphi_\sharp \alpha$ is an isomorphism.
To conclude, observe that $\varphi_\sharp$ is fully faithful.
\end{proof}

\begin{prop}
\label{restriction.12}
Let $\{g_i\colon S_i\to S\}_{i\in I}$ be a family of morphisms in $\lSch/B$ producing a dividing Nisnevich covering sieve.
Then the family of functors $\{g_i^*\}_{i\in I}$ is conservative.
\end{prop}
\begin{proof}
Since $\varphi_\sharp$ is fully faithful,
\cite[Proposition 2.2.3]{logGysin} implies that the family $\{g_i^*\varphi_\sharp\}_{i\in I}$ is conservative.
Together with \eqref{restriction.3.1},
we see that the family $\{\varphi_\sharp g_i^*\}_{i\in I}$ is conservative.
Use the fact that $\varphi_\sharp$ is fully faithful to conclude.
\end{proof}

\begin{prop}
\label{restriction.13}
Let
\[
Q
:=
\begin{tikzcd}
X'\ar[d,"f'"']\ar[r,"g'"]&
X\ar[d,"f"]
\\
S'\ar[r,"g"]&
S
\end{tikzcd}
\]
be a strict Nisnevich distinguished square in $\lSch/B$.
Then the induced square
\begin{equation}
\label{restriction.13.1}
\begin{tikzcd}
h_\sharp h^*\ar[d,"ad'"']\ar[r,"ad'"]&
g_\sharp g^*\ar[d,"ad'"]
\\
f_\sharp f^*\ar[r,"ad'"]&
\id
\end{tikzcd}
\end{equation}
is cartesian,
where $h:=gf'=fg'$.
\end{prop}
\begin{proof}
It suffices to show that the square obtained by applying \eqref{restriction.13.1} to $M_S(U)$ is cartesian for every $U\in \eSm/S$.
This holds since $Q\times_S U$ is a strict Nisnevich distinguished square.
\end{proof}

\begin{prop}
\label{restriction.9}
Let $f\colon X\to S$ be a log smooth morphism in $\lSch/B$ with $S\in \Sch/B$.
Consider the induced cartesian square of the form \eqref{outline.0.1}.
Then the square
\[
\begin{tikzcd}
f^*\ar[d,"ad"']\ar[r,"ad"]&
i_*i^*f^*\ar[d,"ad"]
\\
p_*p^*f^*\ar[r,"ad"]&
q_*q^*f^*
\end{tikzcd}
\]
is cartesian, where $q=pi'=ip'$.
\end{prop}
\begin{proof}
Apply $\varphi^*(-)\varphi_\sharp$ to \eqref{outline.0.2},
and use the fact that $\varphi_\sharp$ commutes with $(-)^*$ and $\varphi^*$ commutes with $(-)_*$.
\end{proof}

\begin{prop}
\label{restriction.8}
For $X\in \lSch/B$ with a chart $\N$,
consider the induced cartesian square of the form \eqref{outline.0.1}.
Then the square
\[
\begin{tikzcd}
\id\ar[d,"ad"']\ar[r,"ad"]&
i_*i^*\ar[d,"ad"]
\\
p_*p^*\ar[r,"ad"]&
q_*q^*
\end{tikzcd}
\]
is cartesian, where $q=pi'=ip'$.
\end{prop}
\begin{proof}
Apply $\varphi^*(-)\varphi_\sharp$ to \eqref{outline.0.3},
and use the fact that $\varphi_\sharp$ commutes with $(-)^*$ and $\varphi^*$ commutes with $(-)_*$.
\end{proof}

\begin{df}
\label{generation.5}
Suppose $S\in \lSch/B$.
Let
\[
\ul{\sT}_{/S}^\ex
\colon
(\Sch/\ul{S})^{\op}
\to
\CAlg(\PrL)
\]
be the functor given by
\[
\ul{\sT}_{/S}^\ex(V)
:=
\sT^\ex(V\times_{\ul{S}}S)
\]
for all $V\in \Sch/\ul{S}$.

For every $\infty$-category $\cC$,
let $\Ho(\cC)$ be its homotopy category.
Observe that the functor $\Ho(\ul{\sT}_{/S}^\ex)$ obtained by taking $\Ho$ objectwise is a motivic triangulated category in the sense of \cite[Definition 2.4.25]{CD19}.
\end{df}

\begin{prop}
\label{generation.6}
Let
\[
\begin{tikzcd}
X'\ar[d,"f'"']\ar[r,"g'"]&
X\ar[d,"f"]
\\
S'\ar[r,"g"]&
S
\end{tikzcd}
\]
be a cartesian square in $\lSch/B$ such that $f$ is strict compactifiable smooth and $g$ is strict proper.
Then the natural transformation
\[
f_\sharp g_*'
\xrightarrow{Ex}
g_*f_\sharp'
\]
is an isomorphism.
\end{prop}
\begin{proof}
Using $\Ho(\ul{\sT}_{/S}^\ex)$,
this follows from the proof of \cite[Theorem 2.4.26(3)]{CD19}.
\end{proof}

\begin{prop}
\label{generation.7}
Let $f\colon X\to S$ be a strict smooth compactifiable morphism in $\lSch/B$, and let $i$ be its section.
Then the functor $f_\sharp i_*$ is an equivalence of $\infty$-categories.
\end{prop}
\begin{proof}
Apply \cite[Corollary 2.4.14]{CD19} to $\Ho(\ul{\sT}_{/S}^\ex)$.
\end{proof}

\begin{df}
\label{generation.8}
Let $f\colon X\to S$ be an exact log smooth proper morphism in $\lSch/B$.
Consider the induced commutative diagram
\[
\begin{tikzcd}
X\ar[r,"a"]&
X\times_S X\ar[d,"p_1"']\ar[r,"p_2"]&
X\ar[d,"f"]
\\
&
X\ar[r,"f"]&
S,
\end{tikzcd}
\]
where $a$ is the diagonal morphism, and $p_1$ and $p_2$ are the first and second projections.
Then we have the composite natural transformation
\[
\mathfrak{p}_f
\colon
f_\sharp
\xrightarrow{\simeq}
f_\sharp p_{1*}a_*
\xrightarrow{Ex}
f_* p_{2\sharp} a_*.
\]
\end{df}

\begin{prop}
\label{generation.9}
Let $f$ be a strict smooth proper morphism in $\lSch/B$.
Then $\mathfrak{p}_f$ is an isomorphism.
\end{prop}
\begin{proof}
Apply \cite[Theorem 2.4.26(3)]{CD19} to $\Ho(\ul{\sT}_{/S}^\ex)$
\end{proof}

\begin{prop}
\label{restriction.10}
Let
\[
\begin{tikzcd}
X'\ar[r,"g'"]\ar[d,"f'"']&
X\ar[d,"f"]
\\
S'\ar[r,"g"]&
S
\end{tikzcd}
\]
be a cartesian square in $\lSch/B$ such that $f$ is strict proper.
Then the natural transformation
\[
Ex\colon
g^*f_*
\to
f_*'g'^*
\]
is an isomorphism.
\end{prop}
\begin{proof}
We first treat the case when $f$ is a strict closed immersion.
Let $j\colon U\to S$ be the open complement of $f$,
and let $j'\colon U\times_S S'\to S'$ be the pullback of $j$.
Since the pair of functors $(f'^*,j'^*)$ is conservative by (Loc),
it suffices to show that
\[
f'^*g^*f_*
\to
f'^*f_*'g'^*,
\text{ }
j'^*g^*f_*
\to
j'^*f_*'g'^*
\]
are isomorphisms.
The first one is an isomorphism by (Loc),
and the second one is an isomorphism by ($\eSm$-BC).

Next,
we deal with the case when $f$ is strict smooth proper.
Consider the induced commutative diagram with cartesian squares
\[
\begin{tikzcd}
X'\ar[d,"g'"']\ar[r,"a'"]&
X'\times_{S'}X'\ar[d,"g''"]\ar[r,"p_2'"]&
X'\ar[d,"g'"]
\\
X\ar[r,"a"]&
X\times_S X\ar[r,"p_2"]&
X,
\end{tikzcd}
\]
where $a$ is the diagonal morphism, and $p_2$ is the second projection.
Let $p_1\colon X\times_S X\to X$ and $p_1'\colon X'\times_{S'} X'\to X'$ be the first projections.
Observe that $a$ is a strict closed immersion.
The diagram
\[
\begin{tikzcd}[column sep=small, row sep=small]
f_\sharp'g'^*\ar[rrr,"Ex"]\ar[d,"\simeq"']\ar[rd,"\simeq"]&
&
&
g^*f_\sharp\ar[d,"\simeq"]
\\
f_\sharp'p_{1*}'a_*'g'^*\ar[dd,"Ex"']\ar[rd,"Ex"',leftarrow]&
f_\sharp'g'^*p_{1*}a_*\ar[d,"Ex"]\ar[rr,"Ex"]&
&
g^*f_\sharp p_{1*}a_*\ar[dd,"Ex"]
\\
&
f_\sharp'p_{1*}'g''^*a_*\ar[d,"Ex"]
\\
f_*'p_{2\sharp}'a_*'g'^*\ar[r,"Ex",leftarrow]&
f_*'p_{2\sharp}'g''^*a_*\ar[r,"Ex^{-1}",leftarrow]&
f_*'g'^*p_{2\sharp}a_*\ar[r,"Ex",leftarrow]&
g^*f_*p_{2\sharp}a_*
\end{tikzcd}
\]
commutes,
where the arrow $Ex^{-1}$ is obtained by ($\eSm$-BC).
Hence the diagram
\begin{equation}
\label{restriction.10.1}
\begin{tikzcd}
f_\sharp'g'^*\ar[rrr,"Ex"]\ar[d,"\mathfrak{p}_{f'}"']&
&
&
g^*f_\sharp\ar[d,"\mathfrak{p}_f"]
\\
f_*'p_{2\sharp}'a_*'g'^*\ar[r,"Ex",leftarrow]&
f_*'p_{2\sharp}'g''^*a_*\ar[r,"Ex^{-1}",leftarrow]&
f_*'g'^*p_{2\sharp}a_*\ar[r,"Ex",leftarrow]&
g^*f_*p_{2\sharp}a_*
\end{tikzcd}
\end{equation}
commutes.
The upper and lower middle horizontal arrows are isomorphisms by ($\eSm$-BC),
and the vertical arrows are isomorphisms by Proposition \ref{generation.9}.
The lower left arrow is an isomorphism by the above special case,
so the lower right arrow is an isomorphism.
Together with Proposition \ref{generation.7},
we see that $g^*f_* \xrightarrow{Ex} f_*'g'^*$ is an isomorphism if $f$ is strict proper smooth.

For general $f$,
argue as in the proof of \cite[Proposition 2.3.11(2)]{CD19}.
\end{proof}

\begin{prop}
\label{restriction.14}
Let $f\colon X\to S$ be a strict proper morphism in $\lSch/B$.
Then the natural morphism
\[
Ex\colon
f_*\cF \otimes \cG
\to
f_*(\cF \otimes f^*\cG)
\]
is an isomorphism for $\cF\in \sT^\ex(X)$ and $\cG\in \sT^\ex(Y)$.
\end{prop}
\begin{proof}
Apply \cite[Theorem 2.4.26(2)]{CD19} to $\Ho(\ul{\sT}_{/S}^\ex)$.
\end{proof}

\begin{prop}
\label{restriction.11}
Let $f\colon X\to S$ be a morphism in $\lSch/B$.
Then $f_\sharp$ (if $f$ is exact log smooth), $f^*$, and $f_*$ commute with the twist $(n)$ for every integer $n$.
\end{prop}
\begin{proof}
It suffices to show that $f^*$ commutes with twists since the other cases follow by adjunction.
Consider the induced commutative diagram with cartesian squares
\[
\begin{tikzcd}
X\ar[d,"f"']\ar[r,"a'"]&
X\times \P^1\ar[d,"g"]\ar[r,"p'"]&
X\ar[d,"f"]
\\
S\ar[r,"a"]&
S\times \P^1\ar[r,"p"]&
S,
\end{tikzcd}
\]
where $p$ is the projection, and $a$ is the $0$-section.
We have natural isomorphisms
\[
f^*(1)
\simeq
f^*p_\sharp a_*
\simeq
p_\sharp' g^*a_*
\simeq
p_\sharp' a_*'f^*
\simeq
(1)f^*,
\]
where we need ($\eSm$-BC) and Proposition \ref{restriction.10} for the second and third natural isomorphisms.
Hence $f^*$ commutes with $(1)$.
We also have the natural isomorphisms.
\[
(-1)f^*
\simeq
(-1)f^*(1)(-1)
\simeq
(-1)(1)f^*(-1)
\simeq
f^*(-1).
\]
Hence $f^*$ commutes with $(-1)$.
Together with induction, we deduce that $f^*$ commutes with $(n)$ for all integer $n$.
\end{proof}

\subsection{Strict smooth motives}
\label{generation}

Throughout this subsection,
we fix a log motivic $\infty$-category $\sT$.
Let $\sSm$ denote the class of strict smooth morphisms in $\lSch/B$.

\begin{df}
\label{generation.2}
For $S\in \lSch/B$, let $\sT^\st(S)$ be the full subcategory of $\sT(S)$ generated under colimits, shifts, and Tate twists by $M(X)$ for all $X\in \sSm/S$.
\end{df}

\begin{prop}
\label{generation.3}
Let $i\colon Z\to S$ be a strict closed immersion in $\lSch/B$.
Then the functor $i_*\colon \sT^\ex(Z)\to \sT^\ex(S)$ sends $\sT^\st(Z)$ to $\sT^\st(S)$.
\end{prop}
\begin{proof}
We need to show $i_* M_Z(W)\in \sT^\st(S)$ for every $W\in \sSm/Z$.
By Zariski descent, the question is Zariski local on $W$.
Use \cite[Proposition IV.18.1.1]{EGAIVIV} to assume that there exists $X\in \sSm/S$ such that $V\simeq X\times_S Z$.
The cofiber sequence
\[
j_\sharp j^*M_S(X) \xrightarrow{ad'} M_S(X) \xrightarrow{ad} i_*i^*M_S(X)
\]
obtained from (Loc) finishes the proof since $i^*M_S(X)\simeq M_Z(W)$.
\end{proof}

For an fs monoid $P$ and an ideal $I$ of $P$,
we set $\A_{P,I}:=\Spec(P\mapsto \Z[P]/(I))$.
If $P$ is sharp,
then $\pt_P:=\A_P\times_{\ul{\A_P}}\{0\}\simeq \A_{P,P^+}$,
where $0$ is the origin of $\ul{\A_P}$,
and $P^+$ is the ideal of non-units of $P$.

\begin{prop}
\label{generation.1}
Let $f\colon X\to S$ be a saturated log smooth morphism in $\lSch/B$.
Then $M_S(X)\in \sT^\st(S)$.
\end{prop}
\begin{proof}
We proceed by induction on
\[
r:=\max_{x\in X} \rank(\ol{\cM}_{X/S,x}^\gp),
\]
where $x$ runs over the points of $X$.
If $r=0$, then $f$ is strict, so the claim is trivial.
Hence assume $r>0$.

By Zariski descent, the question is Zariski local on $S$ and $X$.
Use \cite[Propositions A.3, A.4]{divspc} to assume that $f$ admits a neat chart $\theta\colon P\to Q$ at $x\in X$ such that $P$ is neat at $f(x)$, $Q$ is neat at $x$, and the induced morphism $X\to S\times_{\A_P}\A_Q$ is strict smooth.
In particular, $P$ and $Q$ are sharp.

Suppose that $i\colon Z\to S$ is a strict closed immersion and $j\colon U\to S$ is its open complement.
By (Loc),
we have the cofiber sequence
\[
j_\sharp M_U(X\times_S U)
\to
M_S(X)
\to
i_*M_Z(X\times_S Z).
\]
Due to Proposition \ref{generation.3},
we reduce to showing $M_U(X\times_S U)\in \sT^\st(U)$ and $M_Z(X\times_S Z)\in \sT^\st(Z)$.
After suitable stratification,
we reduce to the case when the chart $S\to \A_P$ factors through $\pt_P\simeq \A_{P,P^+}$.

In this case, let $q\colon S\times_{\A_P}\A_{Q,Q^+}\to S$ be the projection, and let $h\colon X\times_{\A_Q}\A_{Q,Q^+} \to S\times_{\A_P}\A_{Q,Q^+}$ be the induced morphism.
Since $\ul{q}$ is an isomorphism and $h$ is strict smooth, there exists a unique cartesian square
\[
\begin{tikzcd}
X\times_{\A_Q}\A_{Q,Q^+} \ar[r]\ar[d,"h"']&
V\ar[d,"g"]
\\
S\times_{\A_P} \A_{Q,Q^+}\ar[r,"q"]&
S.
\end{tikzcd}
\]
The morphism $g$ is automatically strict smooth.

Let $a\colon S\times_{\A_P}\A_{Q,Q^+}\to S\times_{\A_P}\A_Q$ be the induced closed immersion, let $b$ be its open complement, and let $p\colon S\times_{\A_P}\A_Q\to S$ be the projection.
The localization property yields a cofiber sequence
\begin{equation}
\label{generation.1.2}
p_\sharp b_\sharp b^* p^*M_S(V)
\to
p_\sharp p^* M_S(V)
\to
p_\sharp a_*a^*p^*M_S(V).
\end{equation}
Let $G$ be a maximal $\theta$-critical face of $Q$.
Proposition \ref{restriction.6} yields an isomorphism
\[
p_\sharp p^* M_S(V)
\simeq
M_S(V\times_S \A_{Q_G}).
\]
Since $V\times_S \A_{Q_G}$ is strict smooth over $S$, we have $p_\sharp p^*M_S(V)\in \sT^\st(S)$.
By induction, we have $p_\sharp b_\sharp b^*p^*M_S(V)\in \sT^\st(S)$.
Together with the cofiber sequence \eqref{generation.1.2}, we have $p_\sharp a_*a^*p^*M_S(V)\in \sT^\st(S)$.

The localization property yields another cofiber sequence
\begin{equation}
\label{generation.1.3}
p_\sharp b_\sharp b^* M_{S\times_{\A_P}\A_Q}(X)
\to
p_\sharp M_{S\times_{\A_P}\A_Q}(X)
\to
p_\sharp a_* a^* M_{S\times_{\A_P}\A_Q}(X).
\end{equation}
We have isomorphisms
\[
a^* M_{S\times_{\A_P}\A_Q}(X)
\simeq
M_{S\times_{\A_P}\A_{Q,Q^+}}(X\times_{\A_Q}\A_{Q,Q^+})
\simeq
a^*p^*M_S(V).
\]
This implies $p_\sharp a_*a^*M_{S\times_{\A_P}\A_Q}(X)\in \sT^\st(S)$.
By induction, we have
\[
p_\sharp b_\sharp b^*M_{S\times_{\A_P}\A_Q(X)}(X)\in \sT^\st(S).
\]
Together with the cofiber sequence \eqref{generation.1.3}, we have $p_\sharp M_{S\times_{\A_P}\A_Q}(X)\simeq M_S(X)\in \sT^\st(S)$.
\end{proof}

\begin{prop}
\label{generation.4}
Suppose $S\in \lSch/B$.
Then $\sT^\st(S)$ is the full subcategory of $\sT^\ex(S)$ generated under colimits, shifts, and Tate twists by $v_*v^*\unit$ for all strict projective morphisms $v\colon S'\to S$.
\end{prop}
\begin{proof}
Apply \cite[Lemma 2.2.23]{Ayo071} or \cite[Proposition 4.2.13]{CD19} to $\Ho(\ul{\sT}_{/S}^\ex)$ in Definition \ref{generation.5}.
\end{proof}

\subsection{Kummer smooth motives}
\label{ksm}

Throughout this subsection,
we fix a log motivic $\infty$-category $\sT$.
Let $\kSm$ denote the class of Kummer and log smooth (Kummer smooth for short) morphisms in $\lSch/B$.
The purpose of this subsection is to prove a variant of Proposition \ref{generation.1} for Kummer smooth motives under the assumption of strict \'etale descent.

\begin{df}
\label{ksm.2}
For $S\in \lSch/B$, let $\sT^\ksm(S)$ be the full subcategory of $\sT(S)$ generated under colimits, shifts, and Tate twists by $M(X)$ for all Kummer smooth $X\to S$.
\end{df}

\begin{prop}
\label{ksm.3}
Let $i\colon Z\to S$ be a strict closed immersion in $\lSch/B$.
Assume that $\sT$ satisfies strict \'etale descent.
Then the functor $i_*\colon \sT^\ex(Z)\to \sT^\ex(S)$ sends $\sT^\ksm(Z)$ to $\sT^\ksm(S)$.
\end{prop}
\begin{proof}
We need to show $i_* M_Z(W)\in \sT^\ksm(S)$ for every $W\in \kSm/Z$.
By strict \'etale descent, the question is strict \'etale local on $W$ and $S$.
Due to \cite[Theorem IV.3.3.1]{Ogu},
we may assume that there exists a chart $\theta\colon P\to Q$ of $W\to Z$ such that the projection $S\times_{\A_P}\A_Q\to S$ is Kummer smooth and the induced morphism $W\to Z\times_{\A_P} \A_Q$ is strict \'etale.
Use \cite[Proposition IV.18.1.1]{EGAIVIV} to assume that there exists a cartesian square
\[
\begin{tikzcd}
W\ar[d]\ar[r]&
X\ar[d]
\\
Z\times_{\A_P} \A_Q\ar[r]&
S\times_{\A_P}\A_Q
\end{tikzcd}
\]
such that $X\to S\times_{\A_P}\A_Q$ is strict \'etale.
Then we have $X\in \kSm/S$.
The cofiber sequence
\[
j_\sharp j^*M_S(X) \xrightarrow{ad'} M_S(X) \xrightarrow{ad} i_*i^*M_S(X)
\]
obtained from (Loc) finishes the proof since $i^*M_S(X)\simeq M_Z(W)$.
\end{proof}

\begin{prop}
\label{ksm.4}
Let $f\colon X\to S$ be an exact log smooth morphism of fs log schemes,
and let $P$ be an fs chart on $S$.
Then strict \'etale locally on $X$,
there exists an fs chart $\theta\colon P\to Q$ of $f$ such that $\theta$ is injective and locally exact, $\lvert \mathrm{cok}(\theta^\gp)^\mathrm{tor}\rvert$ is invertible in $X$,
and the induced morphism $X\to S\times_{\A_P} \A_Q$ is strict \'etale.
\end{prop}
\begin{proof}
Consider the chart $\theta$ obtained by \cite[Theorem IV.3.3.1]{Ogu}.
By \cite[Lemma 1.3]{MR2604925},
$\theta$ is $\Q$-integral.
This implies that $\theta$ is locally exact by \cite[Theorem I.4.7.7]{Ogu}.
\end{proof}

\begin{prop}
Suppose $S\in \lSch/B$.
Assume that $\sT$ satisfies strict \'etale descent.
Then $\sT^\ksm(S)\simeq \sT^\ex(S)$.
\end{prop}
\begin{proof}
Let $f\colon X\to S$ be an exact log smooth morphism.
We proceed by induction on
\[
r:=\max_{x\in X} \rank(\ol{\cM}_{X/S,x}^\gp)
\]
to show $M_S(X)\in \sT^\ksm(S)$,
where $x$ runs over the points of $X$.
If $r=0$, then $f$ is strict, so the claim is trivial.

By strict \'etale descent, the question is strict \'etale local on $S$ and $X$.
Hence we may assume that we have a chart $\theta$ as in Proposition \ref{ksm.4}, and we may also assume that $\lvert \mathrm{cok}(\theta^\gp)^\mathrm{tor}\rvert$ is invertible in $S$.
Then the projection $S\times_{\A_P} \A_Q\to S$ is exact log smooth.

As in the proof of Proposition \ref{generation.1},
using (Loc) and Proposition \ref{ksm.3},
we reduce to the case where the chart $S\to \A_P$ factors through $\A_{P,P^+}$.

Let $P'$ be the submonoid of $Q$ consisting of elements $p\in Q$ such that $np\in P+Q^*$ for some integer $n>0$.
Gordon's lemma \cite[Theorem I.2.3.19]{Ogu} implies that $P'$ is finitely generated and hence an fs monoid.
Let $\theta'\colon P'\to Q$ be the inclusion homomorphism.
Then $\theta'$ is logarithmic by construction, i.e., $\theta'^{-1}(Q^*)\simeq Q^*$.
We have $(\ol{P})_\Q\simeq (\ol{P'})_\Q$  by construction.
Also, \cite[Theorem I.4.7.7]{Ogu} implies that $(\ol{P})_\Q\to (\ol{Q})_\Q$ is integral.
Hence $(\ol{P'})_\Q\to (\ol{Q})_\Q$ is integral too,
so $P'\to Q$ is locally exact by \cite[Theorem I.4.7.7]{Ogu}.

Since the projection $S\times_{\A_P} \A_{P'}\to S$ is Kummer smooth,
we can replace $S$ by $S\times_{\A_P} \A_{P'}$ and $P$ by $P'$.
Hence we may assume that $\theta$ is logarithmic.
Note that $P$ is no longer sharp,
but the chart $S\to \A_P$ still factors through $\A_{P,P^+}$.

In this case, let $q\colon S\times_{\A_P} \A_{Q,Q^+}\to S$ be the projection.
Since $\theta$ is logarithmic,
we have
\[
\ul{\A_{P,P^+}}\simeq \ul{\A_{P^*}}\simeq \ul{\A_{Q^*}}\simeq \ul{\A_{Q,Q^+}}.
\]
Hence $\ul{q}$ is an isomorphism.
Now, continue as in the proof of Proposition \ref{generation.1} to conclude.
\end{proof}

\section{Support property}
\label{supp}

We fix a log motivic $\infty$-category $\sT$ throughout this section.

\begin{df}
A proper morphism $f\colon X\to S$ in $\lSch/B$ satisfies the \emph{support property} if for every cartesian square
\[
\begin{tikzcd}
V\ar[r,"j'"]\ar[d,"g"']&
X\ar[d,"f"]
\\
U\ar[r,"j"]&
S
\end{tikzcd}
\]
such that $j$ is an open immersion, the natural transformation $Ex\colon j_\sharp g_*\to f_*j_\sharp'$ is an isomorphism.
\end{df}

The purpose of this section is to prove that every proper morphism satisfies the support property for $\sT^\ex$,
which is Theorem \ref{nonversupp.13}.

The organization of the proof is as follows.
In \S \ref{basic},
we review the basic properties of the support property.
In \S \ref{semisupp},
we introduce the notion of the strictly universal support property,
which is used throughout this section.
This is a technical definition that extends the support property to non-proper morphisms.
In \S \ref{proj},
we introduce the notion of the vertically universal support property,
which plays only a restricted role.
We also show that the vertically universal support property implies the projection formula.
In \S \ref{versupp},
we show that some (not all) vertical morphisms satisfy the strictly or vertically universal support property.
The crucial input for this is \cite[Theorem 5.2.11]{logGysin},
which shows that a certain compactification of the multiplication morphism $\A_{\N^2}\to \A_\N$ satisfies the purity for $\sT$ (not $\sT^\ex$).
In \S \ref{nonversupp},
we show that some (not all) morphisms to the standard log point $\pt_\N$ satisfy the strictly universal support property.
To conclude the proof of the support property in \S \ref{proof},
we need the last two axioms of log motivic $\infty$-categories.
The verifications of these axioms for $\SH$ is due to \cite[Theorems 4.30, 4.31]{logA1}.

Hence \cite{logA1}, \cite{divspc}, and \cite{logGysin} are all needed to prove the support property.

\subsection{Basic properties of the support property}
\label{basic}

\begin{prop}
\label{basic.1}
Let
\[
\begin{tikzcd}
W\ar[r,"i'"]\ar[d,"h"']&
X\ar[d,"f"]\ar[r,leftarrow,"j'"]&
V\ar[d,"g"]
\\
Z\ar[r,"i"]&
S\ar[r,leftarrow,"j"]&
U
\end{tikzcd}
\]
be a commutative diagram in $\lSch/B$ with cartesian square such that $f$ is proper, $i$ is a strict closed immersion, and $j$ is the open complement of $i$.
Then the following three conditions are equivalent:
\begin{enumerate}
\item[\textup{(1)}] $Ex\colon j_\sharp g_*\to f_*j_\sharp'$ is an isomorphism.
\item[\textup{(2)}] $Ex\colon i^*f_* \to h_*i'^*$ is an isomorphism.
\item[\textup{(3)}] $i^*f_*j_\sharp'\simeq 0$.
\end{enumerate}
If $g$ is an isomorphism,
then these conditions are equivalent to the following condition:
\begin{enumerate}
\item[\textup{(4)}]
The square
\[
Q
:=
\begin{tikzcd}
\id\ar[d,"ad"']\ar[r,"ad"]&
i_*i^*\ar[d,"ad"]
\\
f_*f^*\ar[r,"ad"]&
w_*w^*
\end{tikzcd}
\]
is cartesian,
where $w:=ih=fi'$.
\end{enumerate}
\end{prop}
\begin{proof}
Observe that (1) and (2) imply (3) by ($\eSm$-BC).
Assume (3).
This means
\begin{equation}
\label{basic.1.1}
i^*j_\sharp g_*
\simeq
i^*f_*j_\sharp'
\simeq
h_*i'^*j_\sharp'.
\end{equation}

From the commutative square
\[
\begin{tikzcd}
f_*\ar[d,"ad"',"\simeq"]\ar[r,"ad","\simeq"']&
f_*j'^*j_\sharp\ar[d,"Ex",leftarrow,"\simeq"']
\\
j^* j_\sharp f_*\ar[r,"Ex"]&
j^* g_*j_\sharp',
\end{tikzcd}
\]
we have
\begin{equation}
\label{basic.1.2}
j^*j_\sharp g_*\simeq
j^*g_*j_\sharp'.
\end{equation}
We deduce (1) from \eqref{basic.1.1}, \eqref{basic.1.2}, and (Loc).

From the commutative square
\[
\begin{tikzcd}
i^*i_*h_*\ar[r,"ad'","\simeq"']\ar[d,"\simeq"']&
h_*
\\
i^*f_*i_*'\ar[r,"Ex"]&
h_*i'^*i_*',\ar[u,"ad'"',"\simeq"]
\end{tikzcd}
\]
we have
\begin{equation}
\label{basic.1.3}
i^*f_*i_*'
\simeq
h_*i'^*i_*'.
\end{equation}
We deduce (2) from \eqref{basic.1.1}, \eqref{basic.1.3}, and (Loc).
So far,
we have shown that (1), (2), and (3) are equivalent.

From now on,
assume that $g$ is an isomorphism.
By ($\eSm$-BC),
we have
\[
i^*Q j_\sharp'
\simeq
\begin{tikzcd}
0\ar[r]\ar[d]&
0\ar[d]
\\
i^*f_*j_\sharp'\ar[r]&
0.
\end{tikzcd}
\]
Hence (4) implies (3).

Assume (2).
Use (Loc) and (2) to show that $i^*Q$ is cartesian.
Use ($\eSm$-BC) and the assumption that $g$ is an isomorphism to show that $j^*Q$ is cartesian.
Together with (Loc),
we deduce (4).
\end{proof}

\begin{prop}
\label{semisupp.2}
Let $Y\xrightarrow{g} X\xrightarrow{f} S$ be proper morphisms in $\lSch/B$.
If $f$ and $g$ satisfy the support property, then $fg$ satisfies the support property.
\end{prop}
\begin{proof}
Let $j\colon U\to S$ be an open immersion in $\lSch/B$.
We have a commutative diagram with cartesian squares
\[
\begin{tikzcd}
W\ar[r,"g'"]\ar[d,"j''"']&
V\ar[r,"f'"]\ar[d,"j'"]&
U\ar[d,"j"]
\\
Y\ar[r,"g"]&
X\ar[r,"f"]&
S.
\end{tikzcd}
\]
Since the diagram
\[
\begin{tikzcd}
j_\sharp f_*'g_*'\ar[d,"\simeq"']\ar[r,"Ex"]&
f_*j_\sharp'g_*\ar[r,"Ex"]&
f_*g_*j_\sharp''\ar[d,"\simeq"]
\\
j_\sharp(f'g')_*\ar[rr,"Ex"]&
&
(fg)_*j_\sharp'',
\end{tikzcd}
\]
commutes, $fg$ satisfies the support property.
\end{proof}

\begin{prop}
\label{semisupp.10}
Let $f\colon X\to S$ be a strict proper morphism in $\lSch/B$.
Then $f$ satisfies the support property.
\end{prop}
\begin{proof}
Apply \cite[Theorem 2.4.26(2)]{CD19} to the functor $\Ho(\ul{\sT}_{/S}^\ex)$ in Definition \ref{generation.5}.
\end{proof}

\subsection{Strictly universal support property}
\label{semisupp}

\begin{df}
\label{semisupp.1}
Let $f\colon X\to S$ be a morphism in $\lSch/B$.
We say that $f$ satisfies the \emph{strictly universal support property} if for every strict morphism $U\to \ul{X}\times_{\ul{S}}S$,
the projection $X\times_{\ul{X}\times_{\ul{S}}S} U\to U$ satisfies the support property.
\end{df}

Observe that $f$ satisfies the strictly universal support property if and only if the induced morphism $X\to \ul{X}\times_{\ul{S}}S$ satisfies the strictly universal support property.

\begin{prop}
\label{semisupp.3}
Let $f\colon X\to S$ be a proper morphism in $\lSch/B$.
Then $f$ satisfies the strictly universal support property if and only if for all strict morphism $S'\to S$, the projection $f'\colon X\times_S S'\to S'$ satisfies the support property.
\end{prop}
\begin{proof}
If $f$ satisfies strictly universal support property and $S'\to S$ is a strict morphism, then the projection $f'\colon X\times_S S'\to S'$ has the induced factorization
\[
X\times_S S'\to \ul{X}\times_{\ul{S}} S'\to S'.
\]
The first morphism satisfies the support property since the pullback $\ul{X}\times_{\ul{S}} S'\to \ul{X}\times_{\ul{S}} S$ is strict, and the second morphism satisfies the support property since it is strict proper.
Proposition \ref{semisupp.2} shows that $f'$ satisfies the support property.

Conversely, assume that for every strict morphism $S'\to S$, the projection $f'\colon X\times_S S'\to S'$ satisfies the support property.
We set $Y:=\ul{X}\times_{\ul{S}}S$ for simplicity of notation.
If $U\to Y$ is a strict morphism, then the projection $p'\colon X\times_{Y}U\to U$ admits the induced factorization
\[
X\times_{Y} U
\xrightarrow{r}
X\times_{S} U
\xrightarrow{q}
U.
\]
Since $q$ is a pullback of $f$ and $U\to S$ is strict, our assumption implies that $q$ satisfies the support property.
The diagonal morphism $Y\to Y\times_S Y$ is a strict closed immersion, so its pullback $r$ is also a strict closed immersion.
It follows that $r$ satisfies the support property.
Proposition \ref{semisupp.2} shows that $p'$ satisfies the support property.
\end{proof}

\begin{prop}
\label{semisupp.4}
Let $f\colon X\to S$ be a morphism in $\lSch/B$.
If there exists a Zariski covering $\{v_i\colon V_i\to X\}_{i\in I}$ such that $fv_i$ satisfies the strictly universal support property for every $i\in I$, then $f$ satisfies the strictly universal support property.
\end{prop}
\begin{proof}
Replace $f$ by $X\to \ul{X}\times_{\ul{S}}S$ to assume that $\ul{f}$ is an isomorphism.
Then there exists an open immersion $u_i\colon U_i\to S$ for every $i\in I$ such that $X\times_S U_i\simeq V_i$.
Consider the induced commutative diagram with cartesian squares
\[
\begin{tikzcd}
V_i\ar[r,"\id"]\ar[d,"\id"']&
V_i\ar[r,"g_i"]\ar[d,"v_i"]&
U_i\ar[d,"u_i"]
\\
V_i\ar[r,"v_i"]&
X\ar[r,"f"]&
S.
\end{tikzcd}
\]
By Proposition \ref{restriction.12},
to show that $f$ satisfies the strictly universal support property,
it suffices to show that $g_i$ satisfies the strictly universal support property.
This is a consequence of Proposition \ref{semisupp.3} and the assumption that $fv_i$ satisfies the strictly universal support property.
\end{proof}

\begin{prop}
\label{semisupp.5}
Any strict morphism in $\lSch/B$ satisfies the strictly universal support property.
\end{prop}
\begin{proof}
If $f\colon X\to S$ is a strict morphism in $\lSch/B$, then the induced morphism $p\colon X\to \ul{X}\times_{\ul{S}}S$ is a proper strict morphism.
It follows that any pullback of $p$ satisfies the support property by Proposition \ref{semisupp.10}.
\end{proof}

\begin{prop}
\label{semisupp.7}
Let $f\colon X\to S$ be a morphism in $\lSch/B$, and let $S'\to S$ be a strict morphism.
If $f$ satisfies the strictly universal support property, then the projection $X':=X\times_S S'\to S'$ satisfies the strictly universal support property.
\end{prop}
\begin{proof}
The induced square
\[
\begin{tikzcd}
X'\ar[d]\ar[r]&
\ul{X'}\times_{\ul{S'}}S'\ar[d]
\\
X\ar[r]&
\ul{X}\times_{\ul{S}}S
\end{tikzcd}
\]
is cartesian, which shows the claim.
\end{proof}

\begin{prop}
\label{semisupp.6}
Let $Y\xrightarrow{g} X\xrightarrow{f} S$ be morphisms in $\lSch/B$.
If $f$ and $g$ satisfy the strictly universal support property, then $fg$ satisfies the strictly universal support property.
\end{prop}
\begin{proof}
Consider the induced morphisms
\[
Y\xrightarrow{r} \ul{Y}\times_{\ul{X}}X\xrightarrow{q} \ul{Y}\times_{\ul{S}} S.
\]
By assumption, $r$ satisfies the strictly universal support property.
Since $q$ is a pullback of the induced morphism $X\to \ul{X}\times_{\ul{S}}S$, Proposition \ref{semisupp.7} shows that $q$ satisfies the strictly universal support property.
To finish the proof, check that $rq$ satisfies the strictly universal support property using Proposition \ref{semisupp.2}.
\end{proof}

\begin{prop}
\label{semisupp.9}
Let $Y\xrightarrow{g} X\xrightarrow{f} S$ be morphisms in $\lSch/B$ such that $g$ is proper.
Assume that for every pullback $h$ of $g$ along a strict morphism to $X$, $h^*$ is fully faithful.
If $fg$ satisfies the strictly universal support property, then $f$ satisfies the strictly universal support property.
\end{prop}
\begin{proof}
Replacing $f$ by $X\to \ul{X}\times_{\ul{S}}S$, we may assume that $\ul{f}$ is an isomorphism. The question is preserved by any pullback along a strict morphism to $S$, so we only need to show that $f$ satisfies the support property.

Let $j\colon U\to S$ be an open immersion.
Consider the commutative diagram with cartesian squares
\[
\begin{tikzcd}
W\ar[d,"j''"']\ar[r,"g'"]&
V\ar[d,"j'"]\ar[r,"f'"]&
U\ar[d,"j"]
\\
Y\ar[r,"g"]&
X\ar[r,"f"]&
S.
\end{tikzcd}
\]
We have the commutative diagram
\[
\begin{tikzcd}
j_\sharp f_*'\ar[d,"ad"']\ar[rrrr,"Ex"]&
&
&
&
f_*j_\sharp'\ar[d,"ad"]
\\
j_\sharp f_*'g_*'g'^*\ar[r,"\simeq"]&
j_\sharp (f'g')_*g'^*\ar[r,"Ex"]&
(fg)_*j_\sharp''g'^*\ar[r,"Ex"]&
(fg)_*g^*j_\sharp\ar[r,"\simeq"]&
f_*g_*g^*j_\sharp'.
\end{tikzcd}
\]
The assumption implies that the vertical arrows are isomorphisms.
By the support property for $fg$ and ($\eSm$-BC), the lower horizontal arrows are isomorphisms.
Hence the upper horizontal arrow is an isomorphism.
\end{proof}

\begin{prop}
\label{semisupp.8}
Let $f\colon X\to S$ be a morphism in $\lSch/B$.
If there exists a dividing cover $j\colon U\to X$ such that $fj$ satisfies the strictly universal support property, then $f$ satisfies the strictly universal support property.
\end{prop}
\begin{proof}
This is an immediate consequence of ($\divi$-inv) and Proposition \ref{semisupp.9}.
\end{proof}

\subsection{Projection formula}
\label{proj}

\begin{prop}
\label{proj.2}
Let
\[
\begin{tikzcd}
X'\ar[d,"f'"']\ar[r,"g'"]&
X\ar[d,"f"]
\\
S'\ar[r,"g"]&
S
\end{tikzcd}
\]
be a cartesian square in $\lSch/B$ such that $f$ is proper and $g$ is strict.
Assume that $f$ satisfies the strictly universal support property.
Then the natural transformation
\[
g^*f_*
\xrightarrow{Ex}
f_*'g'^*
\]
is an isomorphism.
\end{prop}
\begin{proof}
The question is Zariski local on $S'$ by ($\eSm$-BC) and Proposition \ref{restriction.12}.
so we may assume that $g$ is quasi-projective.
In this case, $g$ admits a factorization
\[
S'\xrightarrow{i} V\xrightarrow{p} S
\]
such that $p$ is strict smooth and $i$ is strict closed immersion.
By ($\eSm$-BC), we can replace $S'\to S$ by $S'\to V$.
Hence we may assume that $g$ is a strict closed immersion.

Let $h\colon S''\to S$ be the open complement of $g$.
Consider the induced commutative diagram with cartesian squares
\[
\begin{tikzcd}
X'\ar[r,"g'"]\ar[d,"f'"']&
X\ar[d,"f"]\ar[r,leftarrow,"h'"]&
X''\ar[d,"f''"]
\\
S'\ar[r,"g"]&
S\ar[r,leftarrow,"h"]&
S''.
\end{tikzcd}
\]
By (Loc), the two rows in the commutative diagram
\begin{equation}
\label{proj.2.1}
\begin{tikzcd}
g^*f_*h_\sharp'h'^*\ar[d,"Ex"']\ar[r,"ad'"]&
g^*f_*\ar[d,"Ex"]\ar[r,"ad"]&
g^*f_*g_*'g'^*\ar[d,"Ex"]
\\
f_*'g'^*h_\sharp'h'^*\ar[r,"ad'"]&
f_*'g'^*\ar[r,"ad"]&
f_*'g'^*g_*'g'^*
\end{tikzcd}
\end{equation}
are cofiber sequences.
Since $f$ satisfies the support property, we have $h_\sharp f_*'\simeq f_*h_\sharp'$.
Together with $g^*h_\sharp \simeq 0$ and $g'^*h_\sharp'\simeq 0$, we see that the left vertical arrow of \eqref{proj.2.1} is an isomorphism.
By (Loc),
we have $g^*g_*\simeq \id$ and $g'^*g_*'\simeq \id$.
The right vertical arrow of \eqref{proj.2.1} is the composition
\[
g^*f_*g_*'g'^*
\xrightarrow{\simeq}
g^*g_*f_*'g'^*
\xrightarrow{ad^{-1}}
f_*'g'^*
\xrightarrow{ad}
f_*'g'^*g_*'g'^*,
\]
which is an isomorphism.
It follows that the middle vertical arrow of \eqref{proj.2.1} is an isomorphism.
\end{proof}

\begin{df}
Let $f\colon X\to S$ be a morphism in $\lSch/B$.
We say that $f$ satisfies \emph{the vertically universal support property} if for every vertical morphism $S'\to S$,
the projection $X\times_S S'\to S'$ satisfies the strictly universal support property.
\end{df}

\begin{prop}
\label{proj.1}
Let $f\colon X\to S$ be a proper morphism in $\lSch/B$.
If $f$ satisfies the vertically universal support property, then $f$ satisfies the projection formula.
\end{prop}
\begin{proof}
We need to show that the morphism
\[
f_*\cF \otimes M_S(S') \xrightarrow{Ex}
f_*(\cF \otimes f^*M_S(S'))
\]
is an isomorphism for all $\cF\in \sT(X)$ and vertical exact log smooth morphism $S'\to S$.
Proposition \ref{restriction.7} yields a Kummer log smooth morphism $g\colon U\to S$ such that $g^*$ is conservative and the projection $S'\times_S U\to U$ is saturated.
By ($\eSm$-BC),
it suffices to show that the morphism
\[
f_*'g'^*\cF \otimes g^*M_S(S')
\xrightarrow{Ex}
f_*'(g'^*\cF\otimes f'^*g^*M_S(S'))
\]
is an isomorphism,
where $f'\colon X\times_S U\to U$ and $g'\colon X\times_S U\to X$ are the projections.
Since $g$ is vertical by \cite[Proposition I.4.3.3]{Ogu},
we can replace $X\to S\leftarrow S'$ by $X\times_S U\to U \leftarrow S'\times_S U$ to assume that $S'\to S$ is saturated.

In this case, by Propositions \ref{generation.1} and \ref{generation.4}, it suffices to show that the morphism
\[
f_*\cF\otimes v_*\unit
\xrightarrow{Ex}
f_*(\cF\otimes f^*v_*\unit)
\]
is an isomorphism for every strict proper morphism $v\colon V\to S$.
In the commutative diagram
\[
\begin{tikzcd}
f_*\cF\otimes v_*\unit\ar[rr,"Ex"]\ar[dd,"Ex"']&
&
f_*(\cF\otimes f^*v_*\unit)\ar[d,"Ex"]
\\
&
&
f_*(\cF\otimes v_*'f'^*\unit)\ar[d,"Ex"]
\\
v_*(v^*f_*\cF\otimes \unit)\ar[r,"Ex"]&
v_*(f_*'v'^*\cF\otimes \unit)\ar[r,"\simeq"]&
f_*v_*'(v'^*\cF\otimes f'^*\unit),
\end{tikzcd}
\]
the left and lower right vertical arrows are isomorphisms by Proposition \ref{restriction.14},
and the upper right vertical arrow is an isomorphism by Proposition \ref{restriction.10}.
Proposition \ref{proj.2} shows that the lower left horizontal arrow is an isomorphism.
It follows that the upper horizontal arrow is an isomorphism.
\end{proof}

\subsection{Support property for some vertical morphisms}
\label{versupp}

\begin{prop}
\label{versupp.3}
Let $f\colon X\to S$ be a Kummer morphism in $\lSch/B$.
Then $f$ satisfies the strictly universal support property.
\end{prop}
\begin{proof}
We can replace $f$ by the induced morphism $X\to \ul{X}\times_{\ul{S}}S$, so we may assume that $\ul{f}$ is an isomorphism.
Since the class of Kummer morphisms is closed under pullbacks, it suffices to show that $f$ satisfies the support property.

By Proposition \ref{restriction.7}, there exists a Kummer log smooth morphism $g\colon S'\to S$ such that $g^*$ is conservative and the projection $f'\colon X\times_S S'\to S'$ is saturated.
Furthermore, \cite[Corollary I.4.8.12]{Ogu} shows that any Kummer saturated morphism is strict.
Hence $f'$ is strict.
It follows that $f'$ satisfies the support property.
Since $g^*$ is conservative, we can use ($\eSm$-BC) to deduce that $f$ satisfies the support property.
\end{proof}

\begin{prop}
\label{versupp.5}
Let $S\to \A_{\N}\times B$ be a morphism in $\lSch/B$.
Consider the diagonal homomorphism $\N\to \N\oplus \N$, and use this to form the fiber product $S\times_{\A_\N}\A_{\N\oplus \N}$.
Then the projection $g\colon S\times_{\A_\N}\A_{\N\oplus \N}\to S$ satisfies the strictly universal support property.
\end{prop}
\begin{proof}
As in \cite[\S 5.2]{logGysin}, let $W$ be the gluing of
\[
\Spec(\N x\oplus \N y\to \Z[x,y]),
\;
\Spec(\N(xy)\to \Z[xy,y^{-1}]),
\;
\Spec(\N(xy)\to \Z[xy,x^{-1}]).
\]
Consider the morphism $W\to \A_\N=\Spec(\N t\to \Z[t])$ given by the formula $t\mapsto xy$.
We have the induced commutative triangle
\[
\begin{tikzcd}
S\times_{\A_\N}\A_{\N\oplus \N}\ar[r,"j"]\ar[rd,"g"']&
S\times_{\A_\N}W\ar[d,"f"]
\\
&
S,
\end{tikzcd}
\]
where $f$ is the projection, and $j$ is the induced open immersion.
Observe that $f$ is proper exact log smooth.
By Propositions \ref{semisupp.5} and \ref{semisupp.6}, it suffices to show that $f$ satisfies the strictly universal support property.
Let $S'\to S$ be any strict morphism.
By Proposition \ref{semisupp.3}, it suffices to show that the projection $S'\times_{\A_\N}W\to S'$ satisfies the support property.
Replace $S'$ by $S$ to reduce to showing that $f$ satisfies the support property.

Let $i\colon Z\to S$ be a strict closed immersion with its open complement $j\colon U\to S$, and let $j'\colon U\times_{\A_\N}W\to S\times_{\A_\N}W$ be the pullback of $j$.
By Proposition \ref{basic.1}, we only need to show $i^*f_*j_\sharp'\simeq 0$.
Consider the sequence of morphisms
\[
i^*f_*j_\sharp'
\xrightarrow{ad}
\varphi^*\varphi_\sharp i^*f_*j_\sharp'
\xrightarrow{\simeq}
\varphi^*i^*\varphi_\sharp f_*j_\sharp'
\xrightarrow{Ex}
\varphi^*i^*f_* \varphi_\sharp j_\sharp'
\xrightarrow{\simeq}
\varphi^*i^*f_*j_\sharp'\varphi_\sharp.
\]
The first arrow is an isomorphism since $\varphi_\sharp$ is fully faithful.
Hence it suffices to show that $\varphi_\sharp f_*\xrightarrow{Ex} f_*\varphi_\sharp$ is an isomorphism since \cite[Proposition 5.2.12]{logGysin} implies $\varphi^*i^*f_*j_\sharp'\varphi_\sharp\simeq 0$.

We need to show that the composition
\[
\varphi_\sharp f_*
\xrightarrow{ad}
\varphi_\sharp f_*\varphi^*\varphi_\sharp
\xrightarrow{\simeq}
\varphi_\sharp \varphi^*f_*\varphi_\sharp
\xrightarrow{ad'}
f_*\varphi_\sharp
\]
is an isomorphism.
The first arrow is an isomorphism since $\varphi_\sharp$ is fully faithful.
It remains to show that the third arrow is an isomorphism.
For this, it suffices to show that the essential image of $f_*\varphi_\sharp$ is in the essential image of $\varphi_\sharp$.

We set $X:=S\times_{\A_\N}W$, and let $a\colon X\to X\times_S X$ be the diagonal morphism.
Consider the vector bundle $\cE$ associated with the locally free $\cO_X$-module $a^*\Omega_{X\times_S X/X}$, see \cite[Proposition 9.15]{divspc}.
Let $p_n\colon \cE\to S$ be the projection, and let $a_n\colon S\to \cE$ be the $0$-section.
By \cite[Proposition 4.1.2, Theorem 5.2.11]{logGysin}, there exists an isomorphism
\[
f_\sharp \varphi_\sharp
\simeq
f_* p_{n\sharp}a_{n*} \varphi_\sharp.
\]
Use Proposition \ref{restriction.5} and $f_\sharp\varphi_\sharp \simeq \varphi_\sharp f_\sharp$ to obtain an isomorphism
\[
\varphi_\sharp f_\sharp
\simeq
f_*\varphi_\sharp p_{n\sharp}a_{n*}.
\]
Proposition \ref{generation.7} implies that $p_{n\sharp}a_{n*}$ is an equivalence of $\infty$-categories.
It follows that the essential image of $f_*\varphi_\sharp$ is in the essential image of $\varphi_\sharp$.
\end{proof}

\begin{prop}
\label{versupp.1}
Let $f\colon X\to S$ be a vertical morphism in $\lSch/B$.
If $\N$ is a chart of $S$, then $f$ satisfies the strictly universal support property.
\end{prop}
\begin{proof}
By Propositions \ref{semisupp.4} and \ref{semisupp.8} and toric resolution of singularities \cite[Theorem 11.1.9]{CLStoric},
we may assume that $f$ admits a chart $\theta\colon \N\to P$ such that $P$ is an exact chart at some point $x\in X$ and $\ol{P}\simeq \N^n$ for some integer $n\geq 0$.
If $n=0$,
then $f$ is strict,
so $f$ satisfies the strictly universal support property.
Hence assume $n>0$.

By \cite[Lemma I.6.7]{zbMATH07027475},
there is a decomposition $P\simeq P^*\oplus \N^n$.
Since $\theta$ is vertical, we have $\theta(1)=(x,a_1,\ldots,a_n)\in P^*\oplus \N^n$ with $a_1,\ldots,a_n >0$.
Consider the homomorphism
\[
\N
\xrightarrow{\alpha_1}
\N^2
\xrightarrow{\alpha_2}
\cdots
\xrightarrow{\alpha_{n-1}}
\N^n
\]
such that $\alpha_i(x_1,\ldots,x_i):=(x_1,\ldots,x_i,x_i)$ for all integer $1\leq i\leq n-1$.
We have the induced morphisms
\[
X
\xrightarrow{h}
S\times_{\A_{\N}}(\A_{P^*}\times \A_{\N^n})
\xrightarrow{g_{n-1}}
\cdots
\xrightarrow{g_2}
S\times_{\A_{\N}}(\A_{P^*}\times \A_{\N^2})
\xrightarrow{g_1}
S,
\]
whose composition is equal to $f$.
By Proposition \ref{versupp.5}, $g_1,\ldots,g_{n-1}$ satisfy the strictly universal support property.
Since the homomorphism $\N^n\to \N^n$ sending $(x_1,\ldots,x_n)$ to $(a_1x_1,\ldots,a_nx_n)$ is Kummer, $h$ is Kummer.
Proposition \ref{versupp.3} implies that $h$ satisfies the strictly universal support property.
Together with Proposition \ref{semisupp.6}, we deduce that $f$ satisfies the strictly universal support property.
\end{proof}

\begin{prop}
\label{versupp.2}
Let $Y\xrightarrow{g} X\xrightarrow{f} S$ be morphisms in $\lSch/B$.
Assume that $\N$ is a chart of $S$.
If $f$ and $fg$ are vertical, then $g$ satisfies the vertically universal support property.
\end{prop}
\begin{proof}
Replace $X$ by $\ul{Y}\times_{\ul{X}}X$ to assume that $g$ is proper.

By Proposition \ref{semisupp.3}, it suffices to show that the projection $Y\times_X X'\to X'$ satisfies the support property for every vertical morphism $X'\to X$.
Since the class of vertical morphisms are closed under compositions and pullbacks by \cite[Propositions 2.16(1), 2.17]{logA1},
we can replace $Y\to X$ by $Y\times_X X'\to X'$ to reduce to showing that $g$ satisfies the support property.

Let
\[
\begin{tikzcd}
W\ar[d,"g'"']\ar[r,"w"]&
Y\ar[d,"g"]
\\
V\ar[r,"v"]&
X
\end{tikzcd}
\]
be a cartesian square such that $v$ is an open immersion.
We only need to show that the induced morphism
\[
\Hom_{\sT^\ex(X)}(M_X(X')(d)[n],v_\sharp g_*'\cF)
\to
\Hom_{\sT^\ex(X)}(M_X(X')(d)[n],g_*w_\sharp \cF)
\]
is an isomorphism for all $\cF\in \sT(W)$, vertical exact log smooth morphism $p\colon X'\to X$, and $d,n\in \Z$.
By adjunction,
it suffices to show that the induced morphism
\[
\Hom_{\sT^\ex(X')}(\unit(d)[n],p^*v_\sharp g_*'\cF)
\to
\Hom_{\sT^\ex(X')}(\unit(d)[n],p^*g_*w_\sharp \cF)
\]
is an isomorphism.
Since $X'$ and $Y\times_X X'$ are vertical over $S$ by \cite[Propositions 2.16(1), 2.17]{logA1}, we can replace $Y\to X\to S$ by $Y\times_X X'\to X'\to S$ using ($\eSm$-BC).
Hence we reduce to showing that the induced morphism
\[
\Hom_{\sT^\ex(X)}(\unit(d)[n],v_\sharp g_*'\cF)
\to
\Hom_{\sT^\ex(X)}(\unit(d)[n],g_*w_\sharp \cF)
\]
is an isomorphism.
We can replace $S$ by $\ul{X}\times_{\ul{S}}S$ to assume that $\ul{f}$ is an isomorphism.

With all the above reductions,
by adjunction,
it suffices to show that the induced natural transformation
\[
f_*v_\sharp g_*'
\xrightarrow{Ex}
f_*g_*w_\sharp
\]
is an isomorphism.
Since $\ul{f}$ is an isomorphism, there exists a unique cartesian square
\[
\begin{tikzcd}
V\ar[r,"v"]\ar[d,"f'"']&
X\ar[d,"f"]
\\
U\ar[r,"u"]&S,
\end{tikzcd}
\]
and $u$ is an open immersion.
In the commutative diagram
\[
\begin{tikzcd}
u_\sharp f_*'g_*'\ar[d,"\simeq"']\ar[r,"Ex"]&
f_*v_\sharp g_*'\ar[r,"Ex"]&
f_*g_*w_\sharp\ar[d,"Ex"]
\\
u_\sharp (f'g')_*\ar[rr,"\simeq"]&
&
(fg)_*w_\sharp,
\end{tikzcd}
\]
the upper left and lower horizontal arrows are isomorphisms by Proposition \ref{versupp.1}.
It follows that the upper right horizontal arrow is an isomorphism.
\end{proof}

\begin{prop}
\label{nonversupp.1}
Let $f\colon S\times \A_{\N}\times \pt_\N \to S\times \A_{\N}\times \pt_\N$ be the morphism in $\lSch/B$ induced by the homomorphism
\[
\N \oplus \N \to \N\oplus \N,
\;
(a,b) \mapsto (a+b,b).
\]
Then $f$ satisfies the vertically universal support property, and $f^*$ is fully faithful.
\end{prop}
\begin{proof}
Let $g\colon S\times \A_{\N}\times \pt_\N \to S\times \pt_\N$ be the morphism induced by the homomorphism
\[
\N \to \N \oplus \N,
\;
a\mapsto (a,a).
\]
Since $g$ and $gf$ is vertical, Proposition \ref{versupp.2} shows that $f$ satisfies the vertically universal support property.
Together with Proposition \ref{proj.1}, we see that the composition
\[
\cF \otimes f_*f^*\unit
\xrightarrow{Ex}
f_*(f^*\cF\otimes \unit)
\xrightarrow{\simeq}
f_*f^*\cF
\]
is an isomorphism for all $\cF\in \sT(S\times \A_{\N} \times \pt_{\N})$.
Hence it remains to show that the morphism
\[
\unit \xrightarrow{ad} f_*f^*\unit
\]
is an isomorphism.

Let $p\colon S\times \A_{\N}\times \pt_{\N} \to S\times \pt_{\N}$ be the projection, and let $j\colon S\times \G_m \times \pt_{\N}\to S\times \A_{\N} \times \pt_{\N}$ be the induced open immersion.
By ($\ver$-inv),
the natural transformation
\[
\unit \to j_*j^* \unit
\]
is an isomorphism.
Since $jf=j$, we have the commutative square
\[
\begin{tikzcd}
\unit\ar[r,"ad"]\ar[d,"ad"']&
f_*f^*\unit\ar[d,"ad"]
\\
j_*j^*\unit\ar[r,"\simeq"]&
f_*j_*j^*f^*\unit.
\end{tikzcd}
\]
We know that the vertical arrows are isomorphisms.
It follows that the upper horizontal arrow is an isomorphism.
\end{proof}

\subsection{Support property for some non-vertical morphisms}
\label{nonversupp}

\begin{prop}
\label{nonversupp.3}
Let $Y\xrightarrow{g} X\xrightarrow{f} S$ be morphisms in $\lSch/B$ such that $g$ is proper.
Assume $g^*$ is fully faithful.
If $g$ satisfies the vertically universal support property and $fg$ satisfies the strictly universal support property, then $f$ satisfies the strictly universal support property.
\end{prop}
\begin{proof}
By Proposition \ref{semisupp.9}, it suffices to show that for any cartesian square
\[
\begin{tikzcd}
Y'\ar[r,"g'"]\ar[d,"h'"']&
X'\ar[d,"h"]
\\
Y\ar[r,"g"]&
X
\end{tikzcd}
\]
such that $h$ is a strict, the unit $\id \xrightarrow{ad} g_*'g'^*$ is an isomorphism.
Proposition \ref{proj.1} shows that the composition
\[
g_*'g'^*\unit \otimes \cF
\xrightarrow{Ex}
g_*'(g'^*\unit \otimes g'^*\cF)
\xrightarrow{\simeq}
g_*'g'^*\cF
\]
is an isomorphism.
Hence it suffices to show that the induced morphism
\[
\unit \xrightarrow{ad} g_*'g'^*\unit
\]
is an isomorphism.
This admits the factorization
\[
\unit \xrightarrow{\simeq} h^*\unit \xrightarrow{ad} h^*g_*g^*\unit \xrightarrow{Ex} g_*'h'^*g^*\unit \xrightarrow{\simeq} g_*'g'^*\unit.
\]
The second arrow is an isomorphism by assumption,
and the third arrow is an isomorphism by Proposition \ref{proj.2}.
\end{proof}

\begin{prop}
\label{nonversupp.2}
Let $p\colon S\times \A_{\N}\times \pt_{\N}\to S\times \pt_{\N}$ be the projection.
Then $p$ satisfies the strictly universal support property.
\end{prop}
\begin{proof}
Let $f\colon S\times \A_{\N}\times \pt_{\N}\to S\times \A_{\N}\times \pt_{\N}$ be the morphism in Proposition \ref{nonversupp.1}, which satisfies the vertically universal support property, and $f^*$ is fully faithful.
Since $pf$ is vertical, Proposition \ref{versupp.1} shows that $pf$ satisfies the strictly universal support property.
Proposition \ref{nonversupp.3} finishes the proof.
\end{proof}

\begin{prop}
\label{nonversupp.9}
Let $p\colon S\times \A_P\times \pt_\N\to S\times \pt_\N$ be the projection, where $P$ is an fs monoid.
Then $p$ satisfies the strictly universal support property.
\end{prop}
\begin{proof}
The claim is trivial if $\ol{P}$ is trivial.
Assume that $\ol{P}$ is not trivial.
Let $\theta\colon \N\to P$ be any homomorphism such that the face of $P$ generated by $\theta(1)$ is $P$.
Then $\theta$ is vertical.

Let $f\colon S\times \A_P\times \pt_\N\to S\times \A_\N\times \pt_\N$ be the morphism induced by $\theta$, and let $q\colon S\times \A_\N\times \pt_\N\to S\times \pt_\N$ be the morphism induced by the diagonal homomorphism $\N\to \N\oplus \N$.
Then $q$ and $qf$ are vertical, so Proposition \ref{versupp.2} shows that $f$ satisfies the strictly universal support property.
Together with Proposition \ref{nonversupp.2}
we deduce that $p$ satisfies the strictly universal support property.
\end{proof}

\begin{prop}
\label{nonversupp.10}
Let $f\colon X\times \pt_\N\to \ul{X}\times \pt_\N$ be the pullback of the morphism $X\to \ul{X}$ removing the log structure,
where $X\in \lSch/B$.
Then $f$ satisfies the strictly universal support property.
\end{prop}
\begin{proof}
The question is Zariski local on $X$ by Proposition \ref{semisupp.4}.
Hence we may assume that $X$ has a chart $P$.
Then $f$ has the induced factorization
\[
X\times \pt_\N\xrightarrow{h}
\ul{X}\times \A_P\times \pt_\N\xrightarrow{g}
\ul{X}\times \pt_\N.
\]
By Propositions \ref{semisupp.5} and \ref{nonversupp.9}, $h$ and $g$ satisfy the strictly universal support property.
Proposition \ref{semisupp.6} shows that $f=gh$ satisfies the strictly universal support property.
\end{proof}

\subsection{Proof of the support property}
\label{proof}

\begin{lem}
\label{basic.4}
For $S\in \Sch/B$ and $X\in \lSm/S$ such that $\ul{\partial X}$ is a strict normal crossing divisor on $\ul{X}$,
consider the induced cartesian square
\[
\begin{tikzcd}
\ul{\partial X}\times_S S'\ar[d,"f'"']\ar[r,"g'"]&
S'\ar[d,"f"]
\\
\ul{\partial X}\ar[r,"g"]&
S.
\end{tikzcd}
\]
Then the composition
\[
g_*g^*f_*\xrightarrow{Ex}
g_*f_*'g'^*
\xrightarrow{\simeq}
f_*g_*'g'^*
\]
is an isomorphism.
\end{lem}
\begin{proof}
Let $Z_1,\ldots,Z_n$ be the irreducible components of the strict normal crossing divisor $\ul{\partial X}$ on $\ul{X}\in \Sm/S$, and let $g_i\colon Z_i\to S$ and $g_i'\colon Z_i\times_S S'\to S'$ be the induced morphisms.
By strict cdh-descent and induction on $n$, we reduce to showing that the composition
\[
g_{i*}g_i^*f_*
\xrightarrow{Ex}
g_{i*}f_*'g_i'^*
\xrightarrow{\simeq}
f_*g_{i_*}g_i'^*
\]
is an isomorphism.
This holds by ($\eSm$-BC) since $g_i$ is strict smooth.
\end{proof}

\begin{prop}
\label{nonversupp.8}
For $S\in \Sch/B$,
consider the cartesian square
\[
\begin{tikzcd}
S\times \pt_{\N^2} \ar[d,"f'"']\ar[r,"g'"]&
S\times \pt_{\N}\ar[d,"f"]
\\
S\times \pt_{\N} \ar[r,"g"]&
S,
\end{tikzcd}
\]
where $f=g$ is the projection.
Then the natural transformation
\[
Ex\colon g^*f_*\to f_*'g'^*
\]
is an isomorphism.
\end{prop}
\begin{proof}
It suffices to show that the morphism
\begin{equation}
\label{basic.2.1}
\Hom_{\sT^\ex(S\times \pt_\N)}(M(V),g^*f_*\cF)
\to
\Hom_{\sT^\ex(S\times \pt_\N)}(M(V),f_*'g'^*\cF)
\end{equation}
is an isomorphism for all $\cF\in \sT^\ex(S\times \pt_\N)$ and vertical exact log smooth morphism $V\to S\times \pt_\N$.
This question is Zariski local on $V$.
Hence by \cite[Lemma 3.7]{logA1},
we may assume that there exists $Y\in \lSm/(S\times \A_\N)=\eSm/(S\times \A_\N)$ such that $Y\times_{\A_\N}\pt_\N\simeq V$.
Replace $Y$ by $Y-\partial_{S\times \A_\N}Y$ to further assume that $Y$ is vertical over $S\times \A_\N$.
Use toric resolution of singularities \cite[Theorem 11.1.9]{CLStoric} to assume that $\ul{Y}\in \Sm/S$ and $\ul{\partial Y}$ is a strict normal crossing divisor.

Let $p\colon \ul{Y} \to S$, $q\colon Y\to S$, $r\colon \ul{\partial Y}\to S$, and $s\colon \partial Y\to S$ be the induced morphisms.
Then let $p'$, $q'$, $r'$, and $s'$ be their pullbacks along $f$.
By adjunction, \eqref{basic.2.1} can be written as
\begin{equation}
\label{basic.2.2}
\Hom_{\sT^\ex(\partial Y)}(\unit,s_*s^*f_*\cF)
\to
\Hom_{\sT^\ex(\partial Y\times \pt_\N)}(\unit,s_*'s'^*\cF).
\end{equation}
We have commutative squares
\[
Q:=
\begin{tikzcd}
\Hom_{\sT^\ex(\ul{Y})}(\unit,p_*p^*f_*\cF)\ar[d]\ar[r]
&
\Hom_{\sT^\ex(Y)}(\unit,q_*q^*f_*\cF)\ar[d]
\\
\Hom_{\sT^\ex(\ul{\partial Y})}(\unit,r_*r^*f_*\cF)\ar[r]&
\Hom_{\sT^\ex(\partial Y)}(\unit,s_*s^*f_*\cF)
\end{tikzcd}
\]
and
\[
Q':=
\begin{tikzcd}
\Hom_{\sT^\ex(\ul{Y}\times \pt_\N)}(\unit,p_*'p'^*\cF)\ar[d]\ar[r]
&
\Hom_{\sT^\ex(Y\times \pt_\N)}(\unit,q_*'q'^*\cF)\ar[d]
\\
\Hom_{\sT^\ex(\ul{\partial Y}\times \pt_\N)}(\unit,r_*'r'^*\cF)\ar[r]&
\Hom_{\sT^\ex(\partial Y\times \pt_\N)}(\unit,s_*'s'^*\cF).
\end{tikzcd}
\]
We also have a morphism of commutative squares $Q\to Q'$.

By Proposition \ref{restriction.9} (resp.\ Propositions \ref{basic.1} and  \ref{nonversupp.10}), $Q$ (resp.\ $Q'$) is cartesian.
Hence to show that \eqref{basic.2.2} is an isomorphism, 
it suffices to show that the three morphisms from the other corners
\begin{gather*}
\Hom_{\sT^\ex(\ul{Y})}(\unit,p_*p^*f_*\cF)
\to
\Hom_{\sT^\ex(\ul{Y}\times \pt_\N)}(\unit,p_*'p'^*\cF),
\\
\Hom_{\sT^\ex(Y)}(\unit,q_*q^*f_*\cF)
\to
\Hom_{\sT^\ex(Y\times \pt_\N)}(\unit,q_*'q'^*\cF),
\\
\Hom_{\sT^\ex(\ul{\partial Y})}(\unit,r_*r^*f_*\cF)
\to
\Hom_{\sT^\ex(\ul{\partial Y}\times \pt_\N)}(\unit,r_*'r'^*\cF)
\end{gather*}
are isomorphisms.
The first two morphisms are isomorphisms by ($\eSm$-BC) since $Y\in \eSm/S$ and $\ul{Y}\in \Sm/S$.
The third morphism is an isomorphism by Lemma \ref{basic.4}.
\end{proof}

\begin{prop}
\label{nonversupp.11}
Let $f\colon S\times \pt_\N\to S$ be the projection, where $S\in \Sch/B$.
Then $f^*$ is conservative.
\end{prop}
\begin{proof}
Consider the induced commutative diagram
\[
\begin{tikzcd}
S\times \pt_\N\ar[r,"i'"]\ar[d,"g'"']&
S\times \A_\N\ar[d,"g"]
\\
S\ar[r,"i"]&
S\times \A^1\ar[r,"p"]&
S,
\end{tikzcd}
\]
where $p$ is the projection, $i$ is the $0$-section, and $g$ is the morphism removing the log structure.
We set $h:=gi'$.
By Proposition \ref{restriction.8},
the square
\[
\begin{tikzcd}
p_*p^*\ar[r,"ad"]\ar[d,"ad"']&
p_*g_*g^*p^*\ar[d,"ad"]
\\
p_*i_*i^*p^*\ar[r,"ad"]&
p_*h_*h^*p^*
\end{tikzcd}
\]
is cartesian.
Due to ($\A^1$-inv), the left vertical arrow is an isomorphism.
It follows that the right vertical arrow is an isomorphism.
Since $pg$ has a section, $(pg)^*$ is conservative.
To show that $f^*\simeq (ph)^*$ is conservative,
use \cite[Lemma 3.3.1]{logGysin}.
\end{proof}

\begin{prop}
\label{nonversupp.7}
Let $f\colon S\times \pt_\N\to S$ be the projection,
where $S\in \Sch/B$.
Then $f$ satisfies the support property.
\end{prop}
\begin{proof}
Let $i\colon Z\to S$ be a closed immersion in $\Sch/B$.
Consider the induced cube with cartesian squares
\[
\begin{tikzcd}[column sep=small, row sep=small]
&
Z\times \pt_{\N}\times \pt_\N\ar[rr,"q_2"]\ar[ld,"i''"']\ar[dd,"q_1"',near end]&
&
Z\times \pt_\N\ar[ld,"i'"]\ar[dd,"g"]
\\
S\times \pt_\N\times \pt_\N\ar[rr,"p_2",near end,crossing over]\ar[dd,"p_1"']&
&
S\times \pt_\N
\\
&
Z\times \pt_\N\ar[rr,"g",near end]\ar[ld,"i'"']&
&
Z\ar[ld,"i"]
\\
S\times \pt_\N\ar[rr,"f"]&
&
S,\ar[uu,leftarrow,crossing over,"f",near start]
\end{tikzcd}
\]
where $g$ is the projection, $p_1$ and $q_1$ are the first projections, and $p_2$ and $q_2$ are the second projections.

By Proposition \ref{basic.1}, it suffices to show that the natural transformation
\[
i^*f_*
\xrightarrow{Ex}
g_*i'^*
\]
is an isomorphism.
Since $g^*$ is conservative by Proposition \ref{nonversupp.11}, it suffices to show that the natural transformation
\[
g^*i^*f_*
\xrightarrow{Ex}
g^*g_*i'^*
\]
is an isomorphism.

Consider the commutative diagram
\[
\begin{tikzcd}
g^*i^*f_*\ar[d,"\simeq"']\ar[r,"Ex"]&
g^*g_*i'^*\ar[r,"Ex"]&
q_{2*}q_1^*i'^*\ar[d,"\simeq"]
\\
i'^*f^*f_*\ar[r,"Ex"]&
i'^*p_{2*}p_1^*\ar[r,"Ex"]&
q_{2*}i''^*p_1^*.
\end{tikzcd}
\]
Proposition \ref{nonversupp.8} implies that the lower left and upper right horizontal arrows are isomorphisms.
By Propositions \ref{basic.1} and \ref{nonversupp.10}, the lower right horizontal arrow is an isomorphism.
It follows that the upper left horizontal arrow is an isomorphism.
\end{proof}

\begin{prop}
\label{nonversupp.4}
Let $f\colon X\to \ul{X}$ be the morphism removing the log structure,
where $X\in \lSch/B$.
If $X$ has a chart $\N$,
then $f$ satisfies the support property.
\end{prop}
\begin{proof}
Let $i\colon \ul{Z}\to \ul{X}$ be an immersion in $\Sch/B$.
Consider the induced cartesian square
\[
\begin{tikzcd}
Z\ar[r,"i'"]\ar[d,"f'"']&
X\ar[d,"f"]&
\\
\ul{Z}\ar[r,"i"]&
\ul{X}.
\end{tikzcd}
\]
By Proposition \ref{basic.1}, it suffices to show that the natural transformation
\[
i^*f_*
\xrightarrow{Ex}
f_*'i'^*
\]
is an isomorphism.
Let $a\colon \ul{\partial Z}\to \ul{Z}$ and $b\colon Z-\partial Z \to \ul{Z}$ be obvious immersions.
Due to (Loc),
it suffices to show that  the natural transformations
\[
a^*i^*f_*
\xrightarrow{Ex}
a^*f_*'i'^*,
\text{ }
b^*i^*f_*
\xrightarrow{Ex}
b^*f_*'i'^*
\]
are isomorphisms.
Since $f_*'$ commutes with $a^*$ (resp.\ $b^*$) by Propositions \ref{restriction.8} and \ref{basic.1} (resp.\ ($\eSm$-BC)),
we reduce to the following two cases:
\begin{enumerate}
\item[(1)]
$\partial Z=Z$.
\item[(2)]
$\partial Z=\emptyset$.
\end{enumerate}

In the case (1), $i'\colon Z\to X$ factors through $\partial X$.
Consider the induced commutative diagram with cartesian squares
\[
\begin{tikzcd}
Z\ar[r,"h'"]\ar[d,"f'"']&
\partial X\ar[r,"g'"]\ar[d,"f''"]&
X\ar[d,"f"]
\\
\ul{Z}\ar[r,"h"]&
\ul{\partial X}\ar[r,"g"]&
\ul{X}.
\end{tikzcd}
\]
We have the commutative diagram
\[
\begin{tikzcd}
h^*g^*f_*\ar[r,"Ex"]\ar[d,"\simeq"']&
h^*f_*''g'^*\ar[r,"Ex"]&
f_*'h'^*g'^*\ar[d,"\simeq"]
\\
i^*f_*\ar[rr,"Ex"]&
&
f_*'i'^*.
\end{tikzcd}
\]
The upper left horizontal arrow is an isomorphism by Propositions \ref{restriction.8} and \ref{basic.1}.
The upper right horizontal arrow is an isomorphism by Propositions \ref{basic.1} and \ref{nonversupp.7}.
It follows that the lower horizontal arrow is an isomorphism.

In the case (2), $i\colon Z\to X$ factors through the obvious immersion $v\colon Z\to X-\partial X$.
Let $u'\colon X-\partial X\to X$ be the obvious immersion, and we set $u:=fu'$.
The natural transformation
\[
v^*u^*f_*
\xrightarrow{Ex}
v^*u'^*
\]
is an isomorphism by ($\eSm$-BC), which is what we need to show.
\end{proof}

\begin{prop}
\label{nonversupp.12}
Let $f\colon X\to \ul{X}$ be morphism removing the log structure,
where $X\in \lSch/B$.
Then $f$ satisfies the strictly universal support property.
\end{prop}
\begin{proof}
The question is Zariski local on $X$ by Proposition \ref{semisupp.4}.
Hence we may assume that $X$ has a chart $P$.
If $\ol{P}=0$,
then the claim is trivial.
If not,
then choose a homomorphism $\theta\colon \N\to P$ such that the face of $P$ generated by $\theta(1)$ is $P$.
Then $f$ has the induced factorization
\[
X\xrightarrow{r}
\ul{X}\times \A_P
\xrightarrow{q}
\ul{X}\times \A_\N
\xrightarrow{p}
\ul{X},
\]
where $p$ is the projection.
Observe that $q$ is vertical.
By Propositions \ref{semisupp.5}, \ref{versupp.1}, and \ref{nonversupp.4}, $p$, $q$, and $r$ satisfy the strictly universal support property.
Proposition \ref{semisupp.6} implies that $f$ satisfies the strictly universal support property.
\end{proof}

\begin{thm}
\label{nonversupp.13}
Let $f\colon Y\to X$ be a proper morphism in $\lSch/B$.
Then $f$ satisfies the support property.
\end{thm}
\begin{proof}
Let
\[
\begin{tikzcd}
W\ar[d,"f'"']\ar[r,"w"]&
Y\ar[d,"f"]
\\
V\ar[r,"v"]&
X
\end{tikzcd}
\]
be a cartesian square such that $v$ is an open immersion.
We only need to show that the induced morphism
\[
\Hom_{\sT(X)}(M_X(X'),v_\sharp f_*'\cF)
\to
\Hom_{\sT(X)}(M_X(X'),f_*w_\sharp \cF)
\]
is an isomorphism for all $\cF\in \sT(V)$ and vertical exact log smooth morphism $p\colon X'\to X$.
By adjunction,
it suffices to show that the induced morphism
\[
\Hom_{\sT(X')}(\unit,p^*v_\sharp f_*'\cF)
\to
\Hom_{\sT(X')}(\unit,p^*f_*w_\sharp \cF)
\]
is an isomorphism.
Use ($\eSm$-BC) and replace $Y\to X$ by $Y\times_X X'\to X'$ to reduce to showing that the induced morphism
\[
\Hom_{\sT(X)}(\unit,v_\sharp f_*'\cF)
\to
\Hom_{\sT(X)}(\unit,f_*w_\sharp \cF)
\]
is an isomorphism.

Consider the cartesian square
\[
\begin{tikzcd}
V\ar[r,"v"]\ar[d,"p'"']&
X\ar[d,"p"]
\\
\ul{V}\ar[r,"\ul{v}"]&
\ul{X},
\end{tikzcd}
\]
where $p$ is the morphism removing the log structure.
By adjunction,
it suffices to show that the upper right horizontal arrow in the commutative diagram
\[
\begin{tikzcd}
\ul{v}_\sharp p_*'f_*'\ar[r,"Ex"]\ar[d,"\simeq"']&
p_*v_\sharp f_*'\ar[r,"Ex"]&
p_*f_*w_\sharp\ar[d,"\simeq"]
\\
\ul{v}_\sharp (p'f')_*\ar[rr,"Ex"]&
&
(pf)_*w_\sharp
\end{tikzcd}
\]
is an isomorphism.
Hence it suffices to show that $p$ and $pf$ satisfy the support property.

Let $q\colon Y\to \ul{Y}$ be the morphism removing log structure.
Then we have $pf=\ul{f}q$.
By Proposition \ref{semisupp.5}, $\ul{f}$ satisfies the support property.
Proposition \ref{nonversupp.12} shows that $p$ and $q$ satisfy the support property.
Use Proposition \ref{semisupp.2} to conclude.
\end{proof}

As a consequence of the support property,
we obtain a cdh-descent property as follows,
which is a log version of the cdh-descent property.

\begin{cor}
\label{nonversupp.14}
Let
\[
\begin{tikzcd}
Z'\ar[d,"p'"']\ar[r,"i'"]&
X'\ar[d,"p"]
\\
Z\ar[r,"i"]&
X
\end{tikzcd}
\]
be a cartesian square in $\lSch/B$ such that $i$ is a strict closed immersion,
$p$ is proper,
and the induced morphism $p^{-1}(X-i(Z))\to X-i(Z)$ is an isomorphism.
Then the commutative square
\[
\begin{tikzcd}
\id\ar[d,"ad"']\ar[r,"ad"]&
i_*i^*\ar[d]
\\
p_*p^*\ar[r]&
q_*q^*
\end{tikzcd}
\]
is cartesian,
where $q:=ip'=pi'$.
\end{cor}
\begin{proof}
This is an immediate consequence of Proposition  \ref{basic.1} and Theorem \ref{nonversupp.13}.
\end{proof}

\begin{prop}
\label{nonversupp.15}
Let $f\colon S\times \pt_\N\to S$ be the projection, where $S\in \lSch/B$.
Then $f^*$ is conservative.
\end{prop}
\begin{proof}
Argue as in the proof of Proposition \ref{nonversupp.11},
but use Corollary \ref{nonversupp.14} instead of Proposition \ref{restriction.8}.
\end{proof}

\subsection{Base change property}

The purpose of this subsection is to define $f_!$ for compactifiable morphisms $f$ in $\lSch/B$ and to show that $(-)_!$ satisfies a certain base change property with $(-)^*$.
We use the technique of Liu-Zheng developed in \cite{LZ}
For a usage of this technique in the context of six-functor formalism in rigid analytic geometry,
we refer to Mann's thesis \cite[\S A.5]{2206.02022}.

\begin{prop}
\label{base.5}
Let
\[
C:=
\begin{tikzcd}
X'\ar[d,"f'"']\ar[r,"g'"]&
X\ar[d,"f"]
\\
S'\ar[r,"g"]&
S
\end{tikzcd}
\]
be a cartesian square in $\lSch/B$.
If $f$ is proper and $g$ is $\Q$-integral,
then the natural transformation
\begin{equation}
\label{base.5.1}
g^*f_*
\xrightarrow{Ex}
f_*'g'^*
\end{equation}
is an isomorphism.
\end{prop}
\begin{proof}
By Proposition \ref{restriction.7},
there exists a Kummer log smooth morphism $U\to S$ such that the projection $u\colon S'\times_S U\to U$ is saturated and $v^*$ is conservative,
where $v\colon S'\times_S U\to X'$ is the projection.
It suffices to show that $v^*g^*f_*\xrightarrow{Ex} v^*f_*'g'^*$ is an isomorphism.
By ($\eSm$-BC),
we can replace $C$ by $C\times_S U$ to reduce to the case when $g$ is saturated.

The question is Zariski local on $S$ and $S'$ by ($\eSm$-BC) and Proposition \ref{restriction.12}.
Hence we may assume that $S$ admits a neat chart $P$.
By \cite[Theorems III.1.2.7(4), IV.2.5.5(4)]{Ogu},
we may further assume that $g$ admits a neat chart $\theta\colon P\to Q$.
Observe that $Q$ is a neat chart of $S'$ by \cite[Proposition A.3]{divspc}.
In particular, $P$ and $Q$ are sharp.
Observe also that $\theta$ is saturated.
We have the induced factorization $P\to P/F\xrightarrow{\eta} Q$ of $\theta$,
where $F:=\theta^{-1}(0)$.
This induces the factorization
\[
S'\xrightarrow{p}
S\times_{\A_P}\A_Q
\xrightarrow{q}
S\times_{\A_P}\A_{P/F}
\xrightarrow{r}
S
\]
of $g$.
Let $\cF$ be the class of morphisms $g$ such that \eqref{base.5.1} is an isomorphism for every morphism $f$.
By \cite[Proposition I.4.8.5(3)]{Ogu},
$\eta$ is saturated.
Furthermore,
$\eta$ is local.
By \cite[Propositions I.4.2.1(5), I.4.6.3(4)]{Ogu},
$\eta$ is injective and locally exact.
Together with \cite[Theorems I.4.8.14(4), IV.3.1.8]{Ogu},
we see that $q$ is exact log smooth.
Hence $\cF$ contains $q$ by ($\eSm$-BC).
Since $\cF$ is closed under compositions and $p$ and $r$ are strict,
we reduce to the case when $g$ is strict.

The question is Zariski local on $S'$,
so we may assume that $g$ admits a factorization
\[
S'\xrightarrow{i} S\times \P^n \xrightarrow{h} S
\]
for some integer $n\geq 0$ such that $i$ is a strict closed immersion and $h$ is the projection.
By Proposition \ref{basic.1} and Theorem \ref{nonversupp.13},
$i$ is in $\cF$.
Since $h$ is in $\cF$ by ($\eSm$-BC) and $\cF$ is closed under compositions,
we deduce that $g$ is in $\cF$.
\end{proof}

We recall the notation of multisimplicial sets in \cite[\S 1.3]{LZ} as follows.
For a set $I$,
let $\Set_{I \Delta}$ denote the category of $I$-simplicial sets.
For a subset $J$ of $I$,
let $\op_J^I$ (or $\op_J$ for short) denote the partial opposite functor $\Set_{I \Delta}\to \Set_{I \Delta}$ with respect to $J$.
Let $\delta_I^*\colon \Set_{I \Delta}\to \Set_{\Delta}$ (or $\delta^*$ for short) denote the functor induced by the diagonal functor $\mathbf{\Delta}\to \mathbf{\Delta}^I:=\Fun(I,\mathbf{\Delta})$,
and let $\delta_*^I\colon \Set_{\Delta}\to \Set_{I \Delta}$ (or $\delta_*$ for short) denote its right adjoint.

\begin{const}
\label{base.4}
Let $\cC$ be an $\infty$-category,
and let $\{\cE_i\}_{i\in I}$ be a set of classes of morphisms in $\cC$ containing all degenerate edges.
Consider the $I$-simplicial set $\cC_{\{\cE_i\}_{i\in I}}^\cart$ in \cite[Definition 1.3.16]{LZ}.
As noted in \cite[Remark 1.3.13]{LZ},
$\cC_{\{\cE_i\}_{i\in I}}^\cart$ is an $I$-simplicial subset of $\delta_*\cC$.
By adjunction,
we obtain a map of simplicial sets
\begin{equation}
\label{base.4.1}
\delta^*\cC_{\{\cE_i\}_{i\in I}}^\cart
\to
\cC.
\end{equation}
\end{const}

\begin{const}
\label{base.6}
Let $f\colon X\to S$ be a compactifiable morphism in $\lSch/B$.
As observed in \cite[\S 5.4]{MR1457738},
we have the induced factorization
\[
X
\xrightarrow{q}
\ul{X}\times_{\ul{S}} S
\xrightarrow{j}
\ul{V}\times_{\ul{S}}S
\xrightarrow{p}
\ul{S}
\]
such that $j$ is an open immersion,
$p$ is strict proper,
and $\ul{q}$ is an isomorphism.
\end{const}

\begin{const}
\label{base.1}
Consider the following classes of morphisms in $\cC:=\lSch/B$:
\begin{itemize}
\item $\cE_1$ is the class of strict proper morphisms,
\item $\cE_2$ is the class of open immersions,
\item $\cE_3$ is the class of strict compactifiable morphisms,
\item $\cE_4$ is the class of morphisms $f\colon X\to S$ such that $\ul{f}$ is an isomorphism,
\item $\cE_5$ is the class of compactifiable morphisms,
\item $\cE_6$ is the class of $\Q$-integral morphisms.
\end{itemize}
Observe that $\cE_i$ is closed under pullbacks and compositions for $i=1,\ldots,6$.
Furthermore, $\cE_i$ is admissible in the sense of \cite[Definition 1.3.18(5)]{LZ} for $i=1,\ldots,4$,
i.e., if $p$ is a morphism in $\cE_i$ and $q$ is a morphism in $\cC$,
then $pq$ is in $\cE_i$ if and only if $q$ is in $\cE_i$.

We begin with the functor of $\infty$-categories
\[
\sT^\ex
\colon
\cC^{\op} \to \infCat_\infty.
\]
Compose this with the map of simplicial sets
$
\delta^*\op_{\{1,2,4,6\}}\cC_{\cE_1,\cE_2,\cE_4,\cE_6}^\cart
\to
\cC^{\op}
$
obtained by \eqref{base.4.1} to yield a map of simplicial sets
\[
\delta^*\op_{\{1,2,4,6\}}\cC_{\cE_1,\cE_2,\cE_4,\cE_6}^\cart
\to
\infCat_\infty.
\]
By Proposition \ref{base.5},
we can apply the right adjoint version of \cite[Proposition 2.2.4]{LZ} for $J:=\{1,4\}$ to yield a map of simplicial sets
\[
\delta^*\op_{\{2,6\}}\cC_{\cE_1,\cE_2,\cE_4,\cE_6}^\cart
\to
\infCat_\infty
\]
using $\op_{\{1,4\}}\op_{\{1,2,4,6\}}\simeq \op_{\{2,6\}}$.
By Theorem \ref{nonversupp.13} and ($\eSm$-BC),
we can apply \cite[Proposition 2.2.4]{LZ} for $J:=\{2\}$ to yield a map of simplicial sets
\begin{equation}
\label{base.1.1}
\delta^*\op_{\{6\}}\cC_{\cE_1,\cE_2,\cE_4,\cE_6}^\cart
\to
\infCat_\infty
\end{equation}
using $\op_{\{2\}}\op_{\{2,6\}}\simeq \op_{\{6\}}$.
By Construction \ref{base.6},
every morphism in $\cE_3$ admits a factorization $pq$ with $p\in \cE_1$ and $q\in \cE_2$.
By \cite[Theorem 1.5.4]{LZ},
there is a categorical equivalence of simplicial sets
\[
\delta^*\op_{\{6\}}\cC_{\cE_1,\cE_2,\cE_4,\cE_6}^\cart
\to
\delta^*\op_{\{6\}}\cC_{\cE_3,\cE_4,\cE_6}^\cart.
\]
Combine this with \eqref{base.1.1} to yield a map of simplicial sets
\begin{equation}
\label{base.1.2}
\delta^*\op_{\{6\}}\cC_{\cE_3,\cE_4,\cE_6}^\cart
\to
\infCat_\infty
\end{equation}
If $f\colon X\to S$ is a morphism in $\cE_5$,
then $f$ admits the induced factorization
\[
X\xrightarrow{q} \ul{X}\times_{\ul{S}} S \xrightarrow{p} S
\]
with $p\in \cE_3$ and $q\in \cE_4$.
By \cite[Theorem 1.5.4]{LZ} again,
there is a categorical equivalence of simplicial sets
\[
\delta^*\op_{\{6\}}\cC_{\cE_3,\cE_4,\cE_6}^\cart
\to
\delta^*\op_{\{6\}}\cC_{\cE_5,\cE_6}^\cart.
\]
Combine this with \eqref{base.1.2} to yield a map of simplicial sets
\begin{equation}
\label{base.1.3}
\delta^*\op_{\{6\}}\cC_{\cE_5,\cE_6}^\cart
\to
\infCat_\infty.
\end{equation}
By \cite[Example 1.4.29]{LZ}, there is a categorical equivalence of simplicial sets
\begin{equation}
\label{base.1.9}
\delta^*\op_{\{6\}}\cC_{\cE_5,\cE_6}^\cart
\to
\Corr(\cC,\cE_5,\cE_6),
\end{equation}
where we refer to \cite[\S A.2]{logGysin} for the notation $\Corr(\cC,\cE_5,\cE_6)$.
We note that a morphism in the $\infty$-category $\Corr(\cC,\cE_5,\cE_6)$ can be written as a diagram $(Y\xleftarrow{g} X\xrightarrow{f} S)$ with $g\in \cE_5$ and $f\in \cE_6$.
Combine \eqref{base.1.3} and \eqref{base.1.9} to yield a functor of $\infty$-categories
\begin{equation}
\label{base.1.7}
\Corr(\cC,\cE_5,\cE_6)
\to
\infCat_\infty.
\end{equation}
\end{const}

\begin{thm}
\label{base.2}
There exists a functor
\begin{gather*}
(\sT^\ex)_!^*
\colon
\Corr(\lSch/B,\mathrm{compactifiable},\Q\textup{-}\mathrm{integral})
\to
\infCat_\infty
\end{gather*}
sending a morphism $(Y\xleftarrow{g} X \xrightarrow{f} S)$ in the correspondence category to
\[
g_!f^*
\colon
\sT^\ex(S)\to \sT^\ex(Y).
\]
Furthermore, $g_!\simeq g_\sharp$ if $g$ is an open immersion,
and $g_!\simeq g_*$ if $g$ is proper.
\end{thm}
\begin{proof}
Only the last claim requires attention by Construction \ref{base.1}.
If $g\colon X\to Y$ is a proper morphism in $\lSch/B$,
then $g$ admits the induced factorization
\[
X\xrightarrow{i} \ul{X}\times_{\ul{S}}S \xrightarrow{p} S.
\]
Since $\ul{i}$ is an isomorphism and $p$ is strict proper,
we have $p_!\simeq p_*$ and $i_!\simeq i_*$ by Construction \ref{base.1}.
Hence we have $g_!\simeq p_!i_!\simeq p_*i_*\simeq g_*$.
\end{proof}

\begin{cor}
\label{base.3}
Let
\[
\begin{tikzcd}
X'\ar[d,"f'"']\ar[r,"g'"]&
X\ar[d,"f"]
\\
S'\ar[r,"g"]&
S
\end{tikzcd}
\]
be a cartesian square in $\lSch/B$ such that $f$ is compactifiable and $g$ is $\Q$-integral.
Then there is a natural isomorphism
\[
g^*f_!
\simeq
f_!'g'^*.
\]
\end{cor}
\begin{proof}
This is an immediate consequence of Theorem \ref{base.2}.
\end{proof}

With the above notation,
if $g$ is exact log smooth,
then
we have the natural transformation
\begin{equation}
Ex\colon g_\sharp f_!'
\to
f_!g_\sharp'
\end{equation}
given by the composite $g_\sharp f_!'\xrightarrow{ad}g_\sharp f_!'g'^*g_\sharp'\xrightarrow{\simeq} g_\sharp g^* f_!g_\sharp' \xrightarrow{ad'} f_!g_\sharp'$.

The condition in the following statement will be proven in \cite{logsix}.

\begin{prop}
\label{base2.5}
Assume that for every cartesian square
\[
C:=
\begin{tikzcd}
X'\ar[d,"f'"']\ar[r,"g'"]&
X\ar[d,"f"]
\\
S'\ar[r,"g"]&
S
\end{tikzcd}
\]
such that $f$ is $\Q$-integral and proper,
the natural transformation
\[
Ex\colon
g^*f_*
\to
f_*'g'^*
\]
is an isomorphism.
Then there exists a functor
\begin{gather*}
(\sT^\ex)_!^*
\colon
\Corr(\lSch/B,\mathrm{exact}\;\mathrm{compactifiable},\mathrm{all})
\to
\infCat_\infty
\end{gather*}
sending a morphism $(Y\xleftarrow{g} X \xrightarrow{f} S)$ in the correspondence category to
\[
g_!f^*
\colon
\sT^\ex(S)\to \sT^\ex(Y).
\]
\end{prop}
\begin{proof}
Argue as in Construction \ref{base.1}.
\end{proof}

We also obtain the projection formula as follows.

\begin{thm}
For every compactifiable morphism $f\colon X\to S$ in $\lSch/B$,
there exists a natural isomorphism
\[
f_!\cF \otimes \cG
\simeq
f_!(\cF\otimes f^*\cG)
\]
for $\cF\in \sT^\ex(X)$ and $\cG\in \sT^\ex(S)$.
\end{thm}
\begin{proof}
This is an immediate consequence of Proposition \ref{proj.1}, Theorem \ref{nonversupp.13}, and ($\eSm$-PF).
\end{proof}

\appendix

\section{Update\texorpdfstring{ of \cite[Theorem 4.32]{logA1}}{}}

In this section,
we prove the Nisnevich version of \cite[Theorem 4.32]{logA1}.

Let $B\in \Sch$,
and let $\sT$ be a motivic $\infty$-category used in \cite[\S 4.1]{logA1}.
For $Y\in \lSch/\A_{\N,B}$ and integer $d$,
consider
\[
\Theta_d(Y)
:=
\hom_{\sT(\ul{Y-\partial_{\A_{\N,B}}Y})}
(
M(\ul{Y-\partial_{\A_{\N,B}}Y})(d),
M((Y-\partial_{\A_{\N,B}}Y)\times_{\A_{\N,B}}C)
)
\]
in \cite[Construction 4.1]{logA1},
where $\hom$ denotes the Hom spectrum.
We will often omit $d$ in $\Theta_d$ for simplicity.
If $f\colon Y'\to Y$ is a dividing cover in $\A_{\N,B}$,
then \cite[Construction 4.3]{logA1} yields a morphism
\[
f_*
\colon
\Theta(Y')
\to
\Theta(Y).
\]

If $f\colon Y'\to Y$ is a morphism in $\A_{\N,B}$ such that
\begin{equation}
\label{update.0.1}
Y'-\partial_{\A_{\N,B}}Y'\simeq (Y-\partial_{\A_{\N,B}} Y)\times_Y Y',
\end{equation}
then we have the induced morphism
\[
f^*\colon
\Theta(Y)
\to
\Theta(Y').
\]

\begin{exm}
\label{update.3}
Assume that $f$ is of the form $\id \times g\colon Y':=X\times_{\A_{\N,B}}V'\to Y:=X\times_{\A_{\N,B}}V$ for some $X\in \lSm/\A_{\N,B}$ and a morphism $g\colon V'\to V$ in $\lSm/\A_{\N,B}$ such that $V$ and $V'$ are vertical over $\A_{\N,B}$.
Then we have
\[
Y-\partial_{\A_{\N,B}}Y
\simeq
Y-\partial_V Y
\simeq
(X-\partial_{\A_{\N,B}}X)\times_{\A_{\N,B}}V
\]
by \cite[Propositions 2.16(4),(5), 2.17]{logA1}.
We similarly have
\[
Y'-\partial_{\A_{\N,B}}Y'
\simeq
(X-\partial_{\A_{\N,B}}X)\times_{\A_{\N,B}}V'.
\]
Hence \eqref{update.0.1} is satisfied.
\end{exm}

\begin{exm}
\label{update.4}
If $f$ is an open immersion, then $f$ satisfies \eqref{update.0.1}.
\end{exm}

The following removes an assumption in \cite[Proposition 4.4]{logA1}:

\begin{prop}
\label{update.1}
Let $f\colon Y'\to Y$ be a dividing cover in $\lSm/\A_{\N,B}$.
Then the morphism $f_*\colon \Theta(Y')\to \Theta(Y)$ is an isomorphism.
\end{prop}
\begin{proof}
Consider the map $\theta\colon \N\to P:=\N\oplus \N$ given by $1\mapsto (1,n)$ with an integer $n>0$,
and consider the face $F:=\N\oplus 0$ of $P$.
We have an induced commutative diagram
\[
\begin{tikzcd}
Y'\ar[d,"f"']\ar[r,"i'"]&
Y'\times_{\A_{\N,B}}\A_{P,B}\ar[r,leftarrow,"j'"']\ar[d,"f''"]\ar[rr,bend left=15,"q'"]&
Y'\times_{\A_{\N,B}}\A_{P_F,B}\ar[r,"p'"']\ar[d,"f'"]&
Y'\ar[d,"f"]
\\
Y\ar[r,"i"]&
Y\times_{\A_{\N,B}}\A_{P,B}\ar[r,leftarrow,"j"]\ar[rr,bend right=15,"q"']&
Y\times_{\A_{\N,B}}\A_{P_F,B}\ar[r,"p"]&
Y,
\end{tikzcd}
\]
where $f'$ and $f''$ are the pullbacks,
$p$, $p'$, $q$, and $q'$ are the projections,
$j$ and $j'$ are the obvious open immersions,
and $i$ and $i'$ are induced by the first projection $P=\N\oplus \N\to \N$.
We have an induced commutative diagram
\[
\begin{tikzcd}
\Theta(Y)\ar[r,"p^*"']\ar[d,"f_*"']\ar[rr,bend left=15,"q^*"]&
\Theta(Y\times_{\A_{\N,B}}\A_{P_F,B})\ar[d,"f_*'"]\ar[r,leftarrow,"j^*"']&
\Theta(Y\times_{\A_{\N,B}}\A_{P,B})\ar[d,"f_*''"]\ar[r,"i^*"]&
\Theta(Y)\ar[d,"f_*"]
\\
\Theta(Y')\ar[r,"p'^*"]\ar[rr,bend right=15,"q'^*"']&
\Theta(Y'\times_{\A_{\N,B}}\A_{P_F,B})\ar[r,leftarrow,"j'^*"]&
\Theta(Y'\times_{\A_{\N,B}}\A_{P,B})\ar[r,"i'^*"]&
\Theta(Y').
\end{tikzcd}
\]
Here, Examples \ref{update.3} and \ref{update.4} construct the morphisms $(-)^*$.
By \cite[Theorem I.4.9.1]{Ogu},
there exists an integer $n>0$ such that the projection $Y\times_{\A_{\N,B}} \A_{P_F,B}\to \A_{P_F,B}$ is saturated,
so \cite[Proposition 4.4]{logA1} is applicable Zariski locally on $Y\times_{\A_{\N,B}} \A_{P_F,B}$ by \cite[Proposition A.4]{divspc} and \cite[Theorem III.2.5.5]{Ogu}.
Hence $f_*'$ is an isomorphism.

Assume that $j^*$ and $j'^*$ are isomorphisms.
Let $\alpha$ be the composite
\begin{align*}
\Theta(Y')
\xrightarrow{p'^*} &
\Theta(Y'\times_{\A_{\N,B}}\A_{P_F,B})
\\
\xrightarrow{(f_*')^{-1}} &
\Theta(Y\times_{\A_{\N,B}}\A_{P_F,B})
\xrightarrow{(j^*)^{-1}}
\Theta(Y\times_{\A_{\N,B}}\A_{P,B})
\xrightarrow{i^*}
\Theta(Y).
\end{align*}
Then we have
\begin{align*}
f_*\alpha
= &
f_*i^*(j^*)^{-1}(f_*')^{-1}p'^*
\simeq
i'^*f_*''(j^*)^{-1}(f_*')^{-1}p'^*
\\
\simeq &
i'^*(j'^*)^{-1}f_*'(f_*')^{-1}p'^*
\simeq 
i'^*(j'^*)^{-1}p'^*
\simeq
i'^*q^*
\simeq
\id
\end{align*}
and
\[
\alpha f_*
= 
i^*(j^*)^{-1}(f_*')^{-1}p'^* f_*
\simeq
i^*(j^*)^{-1}(f_*')^{-1}f_*'p^*
\simeq
i^*(j^*)^{-1}p^*
\simeq
i^*q^*
\simeq
\id.
\]
It follows that $f_*$ is an isomorphism.

Hence it remains to show that $j^*$ and $j'^*$ are isomorphisms.
Let us focus on $j^*$ since the proof for $j'^*$ is similar.
For $V\in \lSm/\A_{\N,B}$,
consider
\[
\Phi_U(V)
:=
\hom_{\sT(\ul{U\times_{\A_{\N,B}}V})}
(
M(\ul{U\times_{\A_{\N,B}}V})(d)
,
M(\ul{U\times_{\A_{\N,B}}V \times_{\A_{\N,B}}C})
)
\]
with $U:=Y-\partial_{\A_{\N,B}}Y$.
Note that $\Phi_U(V)$ agrees with $\Phi(V)$ in \cite[Construction 4.1]{logA1} if $U=\A_{\N,B}$.
We need to show that the induced morphism
\[
\Phi_U(\A_{P,B})
\to
\Phi_U(\A_{P_F,B})
\]
is an isomorphism.
This can be shown by arguing as in \cite[Lemma 4.5]{logA1}.
\end{proof}

\begin{thm}
\label{update.2}
For $\SH$ and $\DA(-,\Lambda)$,
\cite[Theorem 4.32]{logA1} holds,
where $\Lambda$ is a commutative ring.
\end{thm}
\begin{proof}
With Proposition \ref{update.1} in hand,
the arguments leading to \cite[Theorem 4.32]{logA1} can be written for $\SH$ and $\DA(-,\Lambda)$.
\end{proof}

\bibliography{bib}
\bibliographystyle{siam}

\end{document}